\documentclass[11pt]{amsart}
\usepackage{color}
\usepackage{amsmath,amssymb,amsfonts,amsthm,bm}
\usepackage{graphicx}
\usepackage{float}
\usepackage{enumerate}
\usepackage{appendix}
  \usepackage[shortlabels]{enumitem}

\usepackage{color}

\def\R{\mathbb R}
\def\Z{\mathbb Z}
\def\T{\mathbb T}

\def\N{\mathbb N}

\def\cC{\mathcal C}
\def\cD{\mathcal D}
\def\cE{\mathcal E}
\def\cH{\mathcal H}
\def\cL{\mathcal L}
\def\cM{\mathcal M}
\def\cN{\mathcal N}
\def\cP{\mathcal P}

\def\d{\partial}

\renewcommand\div{{\rm div}\,}
\newcommand\divx{{\rm div}_{\!x}\,}
\renewcommand\lim{{\rm lim}\,}

\renewcommand\sup{{\rm sup}\,}

\renewcommand\log{{\rm log}\,}
\renewcommand\det{{\rm det}\,}

\newcommand{\with}{\quad\!\hbox{with}\!\quad}
\newcommand{\andf}{\quad\!\hbox{and}\!\quad}

\newcommand{\Int}{\displaystyle \int}

\def\df{\delta\!f}

\def\du{\delta\!u}

\def\dw{\delta\!w}

\def\dF{\delta\!F}
\def\dH{\delta\!H}

\def\dS{\delta\!S}
\def\dU{\delta\!U}
\def\dV{\delta\!V}
\def\dX{\delta\!X}
\def\dW{\delta\!W}
\def\dZ{\delta\!Z}
\def\ddj{\dot \Delta_j}

\newtheorem{theorem}{Theorem}[section]
 \newtheorem{corollary}[theorem]{Corollary}
 
 \newtheorem{proposition}[theorem]{Proposition}
 \theoremstyle{definition}
 
 \theoremstyle{remark}
 \newtheorem{remark}[theorem]{Remark}
 
 \numberwithin{equation}{section}

\newcommand{\wh}{\widehat}
\newcommand{\wt}{\widetilde}
\newcommand{\eps}{\varepsilon}
\renewcommand{\div}{\mbox{div}\,}
\newcommand{\divv}{\mbox{div}_v\,}

\begin{document}

\title[]{Fujita-Kato solutions and optimal time decay  for the Vlasov-Navier-Stokes system
in the whole space}
\author{ Rapha\"el Danchin}
\begin{abstract} 
We are concerned with the construction of global-in-time strong solutions for the incompressible Vlasov-Navier-Stokes system
in the whole three-dimensional space. Our primary goal  is to establish that small initial velocities with critical 
Sobolev regularity $H^{1/2}$ and sufficiently well localized initial kinetic distribution functions give rise 
to global and unique solutions. This constitutes an extension of the celebrated result for the incompressible
Navier-Stokes equations (NS) that has been proved  by  Fujita and  Kato in~\cite{FK}. 
Assuming also that the initial velocity is in $L^1,$
 we establish that  the total energy $E_0$ of the system decays  to $0$ with the same rate $t^{-3/2}$
 as  for the weak  solutions of (NS), see \cite{Sch2,W}.  
  
Our results partly rely on the use  of  a higher order 
energy functional  $E_1$ that controls  the regularity $H^1$ of the velocity. This idea 
 seems to originate from the recent paper \cite{LSZ}   by Li,  Shou and  Zhang, devoted to the
 \emph{inhomogeneous} Vlasov-Navier-Stokes system. 
 Here we show that $E_1$  decays with  the  rate $t^{-5/2}$
which, in particular, allows us to prove that the density of the particles has a strong limit when the time goes to infinity.   
\end{abstract}
\keywords{Vlasov-Navier-Stokes, Fujita-Kato solutions, optimal decay rate, global solutions, large time behavior}
\subjclass[2010]{35Q30, 35Q83, 76D05, 76D03}

\maketitle

This paper is concerned with the proof of the existence and uniqueness of global-in-time strong solutions
and the study of large time asymptotics for the following so-called Vlasov-Navier-Stokes system in  $\R^3\times\R^3$:
$$\left\{\begin{array}{l}
\d_tf+v\cdot\nabla_xf+\divv\bigl({(u-v)f} \bigr)=0,\\[1ex]
\d_t u+u\cdot\nabla_x u-\Delta_x u+\nabla_x P=\Int_{\R^3} {(u-v)f}\,dv,\\[1ex]
\div u=0.\end{array}\right.\leqno(VNS)$$
The above system is a coupling between  a kinetic transport equation for the distribution function $f=f(t,x,v)$
and the incompressible Navier-Stokes equations for the fluid velocity  $u=u(t,x)$ and pressure function $P=P(t,x).$
Here $x\in\R^3$ and $v\in\R^3$ stand for the space and kinetic variables, respectively, and 
$t\geq0$ for the time variable.  The kinetic and fluid equations are coupled through  a drag term, the so-called Brinkman force.
\smallbreak
{System (VNS)  is a prototype model for  describing  the dynamics of a spray,  i.e.  a cloud of particles 
that are immersed in a fluid. Such systems were introduced by e.g. Williams in his treatise \cite{Williams} on the 
theory of combustion. As pointed out by O'Rourke in \cite{Rourke},  in the absence of external forces
and for a `thin' spray, that is when there are no interactions between the particles, 
the coupling between the fluid and kinetic equations reduces to the above drag term, and the kinetic equation is linear with respect to $f.$  Further explanations and motivations can be found in e.g. the introduction of  \cite{BDGM}.}  
\medbreak
\smallbreak
For $f\equiv0,$  System (VNS)   reduces to the  incompressible Navier-Stokes equations
$$
\left\{\begin{array}{ll} &\d_t u+u\cdot\nabla u-\Delta u+\nabla P= 0,\\[1ex]
&\div u=0,\end{array}\right.\leqno(NS)$$
and the question arises as to which of the results for (NS) can be extended to (VNS).\smallbreak
Like for (NS), there is an  energy balance associated to (VNS), namely 
\begin{align}\label{eq:energy}
\frac d{dt} E_0+D_0=0\with E_0&:=\frac 12\|u\|_{L^2(\R^3_x)}^2+\frac12 \||v|^2 f\|_{L^1(\R^3_x\times\R^3_v)}\\\andf
D_0&:=\|\nabla u\|_{L^2(\R^3_x)}^2+ \| |u-v|^2f\|_{L^1(\R^3_x\times\R^3_v)}.\nonumber
\end{align}
One can thus expect to have a global finite energy weak solutions  theory similar to that  established by Leray for (NS), in \cite{Leray}. 
Such solutions have indeed been constructed by  Boudin,  Desvillettes, Grandmont and  Moussa
in \cite{BDGM} who leveraging a suitable approximation scheme and 
compactness arguments to prove 
 an analogous of Leray's theorem, for (VNS). More precisely, in the case where the fluid
domain is the three-dimensional torus $\T^3,$ any initial data $(f_0,u_0)$ such that
$$f_0\in L^\infty(\T^3\times\R^3),\quad |v|^2f_0\in L^1(\T^3\times\R^3)\andf u_0\in L^2(\T^3)\with \div u_0=0$$
gives  rise to a global-in-time distributional solution verifying 
\begin{equation}
\label{eq:energyb}E_0(t)+\int_0^t D_0\,d\tau \leq E_{0,0}:= \frac 12\|u_0\|_{L^2(\R^3_x)}^2+\frac12 \||v|^2 f_0\|_{L^1(\R^3_x\times\R^3_v)},
\quad t\in\R_+.\end{equation} 
The result of \cite{BDGM}  also holds in three-dimensional bounded domains \cite{AB}, 
  in the two-dimensional torus, 
and  can be  adapted to the $\R^2$ or $\R^3$ setting if $f_0$ is suitably well localized (see the appendix of \cite{HMMM}).
Furthermore, as  for (NS), in the 2D case, weak solutions are  unique in the energy class (again, see \cite{HMMM}). 
\smallbreak
To continue the analogy with (NS), one can  study whether smoother initial data generate smoother solutions. 
In particular, in the three dimensional case, if assuming that $u_0$ is small  in the Sobolev space $\dot H^{1/2}(\R^3)$ and
that $f_0$ satisfies appropriate conditions,  do we have a global and  unique solution~?
A positive answer to this question would provide a result similar to that of Fujita and Kato \cite{FK} for (NS)
in a three-dimensional bounded domain (and adapted to the $\R^3$ situation by Chemin in \cite{JYC}). 

In the present paper, we focus on the $\R^3$ case. We want to establish the existence and uniqueness of 
global solutions for (VNS) supplemented with an initial velocity satisfying Fujita-Kato type assumption,  and to 
specify  the rate of convergence to $0$ of  the energy $E_0$,  and the long time behavior of   $f.$ 
To our knowledge, proving global existence and uniqueness for critical regularity is new. 
As for the proof of the long time  behavior, it has been studied recently by Han Kwan in \cite{DHK}. 
{Compared to this work, we are here able to consider a larger class of initial data ($u_0$ only needs to be integrable 
and to have the critical regularity $\dot H^{1/2}(\R^3)$), to reach the optimal time decay for $E_0$ and  to 
have a more accurate description of the large-time behavior of the distribution of particles. 
This improvement is partly the result of achieving an optimal decay rate for  the following higher order energy functional:
\begin{equation}\label{def:E1}  E_1:=D_0=\|\nabla u\|_{L^2(\R^3_x)}^2+ \| |u-v|^2f\|_{L^1(\R^3_x\times\R^3_v)}.\end{equation}}


\section{Main results} \label{s:main}

Although our final goal is to investigate  (VNS) supplemented with initial velocity in the critical Sobolev space
$H^{1/2}(\R^3),$ we shall first consider the smoother situation where  $u_0$ belongs to the Sobolev space $H^1(\R^3)$
and   {$E_1(0)$ is small enough (where $E_1$ has been defined in \eqref{def:E1}).
 Then,  $E_1$ can be controlled in terms of the data,  uniformly with respect to time.
In this setting, we will obtain the following result, which is the key to our second one, pertaining to initial velocities in the nonhomogeneous Sobolev space $H^{1/2}(\R^3).$} 
\begin{theorem}\label{thm:main1}Assume that the initial distribution function $f_0$ satisfies 
\begin{equation}\label{eq:f0} f_0\in L^1(\R^3_x\times\R^3_v)\cap L^\infty(\R^3_v\times\R^3_x)\andf
|v|^2f_0 \in  L^1(\R^3_v;L^\infty(\R^3_x))\end{equation}
and that   $u_0$ is  a divergence free vector-field in $L^1(\R^3)\cap  H^1(\R^3).$ 
\smallbreak
There exists a small constant $c_0$ depending only on 
$$\|u_0\|_{L^1(\R^3)},\quad \|f_0\|_{L^1(\R^3\times\R^3)},\quad \|f_0\|_{L^1(\R_v^3;L^\infty(\R^3_x))}
\andf \||v|^2f_0\|_{L^1(\R_v^3;L^\infty(\R^3_x))}$$ such that if \footnote{In all the paper, $\int$ means $\int_{\R^3}.$} 
\begin{equation}\label{eq:smalldata}\|u_0\|_{H^1(\R^3)}^2+  \int\!\!\!\int |v|^2 f_0\,dv\,dx \leq c_0,\end{equation}
then $(VNS)$ admits a unique global solution $(f,u,P)$ satisfying the energy balance:
\begin{equation}\label{eq:energy0}E_0(t)+\int_0^t D_0\,d\tau=E_{0,0}  \quad\hbox{for all } t>0,\end{equation} 
and such that:
\begin{itemize}
\item  $u\in \cC_b(\R_+;H^1)\cap L^2(\R_+;L^\infty),$ 
$\nabla u\in L^1(\R_+;L^\infty),$  
$\nabla^2u\in L^2(\R_+\times \R^3),$ \smallbreak
\item $\nabla P\in L^2(\R_+\times \R^3),$\smallbreak
\item $f\in L^\infty_{loc}(\R_+; L^\infty(\R^3_v\times\R^3_x)),$
$(1+|v|^2)f\in L^\infty(\R_+;L^1(\R^3_v;L^\infty(\R^3_x)))$ and 
\begin{equation}\label{eq:M0}\|f(t)\|_{L^1(\R^3\times \R^3)}
=M_0:=\|f_0\|_{L^1(\R^3\times\R^3)}\quad\hbox{for all }\  t\in\R_+.\end{equation}
\end{itemize}
Furthermore,  the following decay estimates hold for $E_0$ and $E_1$: 
\begin{equation} \label{eq:E1t} 
E_0(t)\leq C_0 t^{-3/2}\andf E_1(t)\leq C_1 t^{-5/2},\end{equation}
for some positive constants  $C_0$ and $C_1$ depending only on  suitable norms of  the initial data. \end{theorem}
\begin{remark} An important  ingredient of our proof will consist in  establishing decay estimates of independent interest involving  $E_0,$  
$E_1,$ $u_t$ and $\nabla u_t,$
see  Corollary \ref{c:final}.\end{remark} 
{\begin{remark}  Under stronger regularity assumptions, Han Kwan proved in \cite{DHK} that $E_0$ decays like
 $t^{-3/2+\varepsilon}$ for all positive $\varepsilon$ but did not consider $E_1.$ 
The exponents in \eqref{eq:E1t} are optimal inasmuch as they correspond to those  for the solution $z$ of  the heat or Stokes equations without source term, namely
$$\|z(t)\|_{L^2}^2\leq Ct^{-3/2}\|z_0\|_{L^1}^2\andf \|\nabla z(t)\|_{L^2}^2\leq Ct^{-5/2}\|z_0\|_{L^1}^2\with z(t)=e^{t\Delta}z_0.$$
These rates are  known to be optimal for the Navier-Stokes equations (see the work by Schonbek \cite{Sch2}).  Hence they are 
optimal  for (VNS), too.  
\end{remark}}
\begin{remark} 
In contrast with the (NS) situation, we do not know how to prove global results without assuming more integrability 
on $u_0$ than that which is given by \eqref{eq:smalldata}.
  In fact, to get a global-in-time uniform control of the following hydrodynamical quantities
(density, momentum and energy density):
\begin{equation}\label{eq:not}
\rho(t,x):=\!\int f(t,x,v)\,dv,\!\quad j(t,x):=\!\int  vf(t,x,v)\,dv,\!\quad m_2(t,x):=\!\int \!|v|^2f(t,x,v)\,dv
\end{equation}
 we need  $\nabla u$ to be  small in $L^1(\R_+;L^\infty).$ 
 This latter property requires sufficient decay for $u$.  
\end{remark}
\begin{remark}\label{r:1}  {The $L^1$ space does not behave well for the Navier-Stokes flow. A more suitable assumption is to have  $u_0$ in the homogeneous Besov space $\dot B^{-3/2}_{2,\infty}(\R^3)$
defined in  Appendix  \ref{s:AppendixB}. 
The proof of decay estimates will be performed in this larger framework
(we have the critical embedding $L^1(\R^3)\hookrightarrow \dot B^{-3/2}_{2,\infty}(\R^3)$). 
The important role of Besov spaces with negative regularity indices for decay estimates has been pointed out in other contexts
(see e.g. the work \cite{XK}  by Xu and Kawashima on partially dissipative systems).} 

{Let us also emphasize that Theorem \ref{thm:main1} holds true if replacing the $L^1$ space by $L^p$ for some $p\in(1,6/5).$
Then, the decay exponents of $E_0$ and $E_1$ are  $\sigma$ and $\sigma+1,$  respectively
where the value of $\sigma$ is given by the critical embedding
of $L^p(\R^3)$ in $\dot B^{-\sigma}_{2,\infty}(\R^3),$ that is, $\sigma=3/p-3/2.$
The same remark will hold for Theorem \ref{thm:main2} below.  }
\end{remark}
As observed in e.g. \cite{HMM}, having decay estimates at hand for the velocity allows to get
some information on the long time asymptotics  of $f,$ $\rho$ and  $j.$ 
To state it precisely, we need to 
introduce   the Wasserstein distance $W_1(\mu,\nu)$ between two measures on $\R^3\times\R^3.$ 
In our context, it can be  defined by:
$$W_1(\mu,\nu):=\sup\Bigl\{ \int\!\!\!\int \phi\, d\mu(x,v)-  \int\!\!\!\int \phi\, d\nu(x,v),\; \phi \in C^{0,1}(\R^3\times\R^3),\,
\|\nabla_{x,v}\phi\|_{L^\infty}=1\Bigr\}\cdotp$$
\begin{corollary}\label{c:was} Under the assumptions of Theorem \ref{thm:main1}, we have for some
constant $C_0$ depending only on the norms of the data:
\begin{equation}\label{eq:conv1}\|(j-\rho u)(t)\|_{L^1}+ W_1(f(t), \rho(t)\otimes\delta_{v=u(t)})\leq C_0t^{-5/4}.\end{equation}
Furthermore,  {$\rho(t)$ converges weakly $*$ in $L^\infty$ when 
$t$ goes to infinity,} and there exists a vector-field $j_\infty$ in $L^1(\R^3)$ such that 
\begin{equation}\label{eq:conv2}\|\rho(t)-\rho_0+{\rm div}\, j_\infty\|_{\dot W^{-1,1}} \leq C_0 t^{-1/4}.\end{equation}
\end{corollary}
\begin{remark}  
This  result has to be compared with that of \cite{DHK} where it is stated that $$W_1(f(t), \rho(t)\otimes\delta_{v=0})
\leq C_1t^{-3/4+\varepsilon}\quad\hbox{for all }\ \varepsilon>0\andf t\geq1.$$ 
Having the exponent  $5/4$ in \eqref{eq:conv1} means that  the monokinetic distribution $ \rho(t)\otimes\delta_{v=u(t)}$ is a better
long-time  approximation of $f$ than $\rho(t)\otimes\delta_{v=0}.$  Having a decay information for $E_1$ and not only for $E_0$ is
the key to this improvement.  \end{remark}

Let us now come to the second aim of the paper: establishing a Fujita-Kato type result for (VNS).
Again, in order to get a global result, we need to assume that $u_0$ has enough integrability. 
For technical reasons, we  also need  $f_0$ to have a better localization than 
in Theorem \ref{thm:main1}: we suppose that  for some $q>5,$  we have\footnote{A similar condition was required in the previous works dedicated to (VNS)
in the whole space and plane (see \cite{DHK} and \cite{HMMM}). Whether it may be avoided is an open question.}
\begin{equation}\label{eq:Nq0}N_q(f_0):=\underset{(x,v)\in\R^3\times\R^3}\sup \langle v\rangle^q f_0(x,v)<\infty
\with \langle v\rangle:=\sqrt{1+|v|^2}.\end{equation} 
\begin{theorem}\label{thm:main2}  Let $u_0$ be a divergence free vector-field in $L^1(\R^3)\cap 
\dot H^{1/2}(\R^3),$ and assume that $f_0$ satisfies \eqref{eq:f0}
and that $N_q(f_0)<\infty$ for some $q>5.$ 
Then, there exists a  constant $c_0$ depending only on $\|u_0\|_{L^1(\R^3)},$ $\|f_0\|_{L^1(\R^3\times\R^3)}$ and $N_q(f_0)$  
such that, if 
\begin{equation}\label{eq:smallness2}\|u_0\|_{H^{1/2}(\R^3)}^2+  \int\!\!\!\int |v|^2 f_0\,dv\,dx  \leq c_0,\end{equation}
then $(VNS)$ admits a unique global solution $(f,u,P)$ satisfying the energy equality \eqref{eq:energy0}, the mass conservation \eqref{eq:M0},
$u\in  L^2(\R_+;L^\infty\cap\dot H^{3/2})\cap \cC_b(\R_+;H^{1/2})$   and
\begin{equation}\label{eq:loglip}\int_0^T \|u\|_{C_{\wt\omega_\eta}}\,dt<\infty  \ \hbox{ for all }\  T\geq0\andf\eta\in(0,1/2),\end{equation}
where 
 we have used the notation 
$$\|u\|_{C_{\wt\omega_\eta}}:= \underset{\substack{x\not= y\\|x-y|\leq1}}\sup\frac{|u(y)-u(x)|}{\wt\omega_\eta(|y-x|)}\with  \wt\omega_\eta(r):= r(1-\log r)^{1-\eta}.$$
Furthermore, the  time decay estimates for $E_0$ and $E_1$ stated in  \eqref{eq:E1t}, as well as the asymptotics
properties \eqref{eq:conv1} and \eqref{eq:conv2} hold true for all $t\geq1.$  
\end{theorem}
\begin{remark}\label{r:main2}
If  we assume that $u_0$ is small in the critical  Besov space $B^{1/2}_{2,1}$ (a `large' subspace of $H^{1/2}$),  
then it suffices to make the hypothesis \eqref{eq:f0} on $f_0.$ In fact, the constructed solution turns out to satisfy 
$\nabla u \in L^1(\R_+;L^\infty)$ so that one can bound $\rho,$ $j$ and $m_2$ as in 
Theorem \ref{thm:main1} without requiring a stronger localization of $f_0,$ see Subsection \ref{ss:critic}. 
\end{remark}

Let us give an overview of the main ideas leading to our  results. 

The first step in proving Theorem \ref{thm:main1} is to establish a priori estimates on $[0,T]$ 
in terms of the data for $E_0,$ $E_1$ and $t\|\d_t u\|_{L^2}^2.$ 
These estimates are \emph{conditional}: it is supposed beforehand that the density $\rho$
is bounded in $[0,T]\times\R^3$ and that  $\nabla u$ is in $L^1(0,T;L^\infty).$ 
 At this stage, no additional integrability is required for $u_0.$ 

Keeping the hypotheses of the first step and assuming in addition that $u_0$ is $L^1,$ 
the second step is dedicated to the proof of optimal time decay estimates for $E_0$ and $E_1.$
In contrast with the recent work \cite{DHK}, we do not use  the Schonbek  Fourier splitting method  of \cite{Sch2}
because  it looks to entail a small loss in the decay rate, in our context. We rather revisit  an approach that seems to originate from 
a work of  Nash  \cite{Nash} on parabolic equations.  
Schematically, the idea goes as follows.  Assume that there exist 
a nonnegative \emph{Lyapunov functional}  $\cL$ and  a \emph{dissipation rate} $\cH$ such that
\begin{equation}\label{eq:lyap} \frac d{dt}\cL +\cH\leq 0\quad\hbox{a.e. on}\  \R_+.\end{equation}
Suppose in addition that we have at hand the   control of a  `lower order' functional~$\cN$:
\begin{equation}\label{eq:N}\cN\leq \cN_0 \quad\hbox{a.e. on}\  \R_+.\end{equation}
Finally assume that $\cL$ is an `intermediate'  functional  between $\cN$ and $\cH$ in the sense that there exist 
 $\theta\in(0,1)$ and $C>0$ such that: 
\begin{equation}\label{eq:interpo}
 \cL\leq C\cH^{\theta}\cN^{1-\theta}.\end{equation} Then, inserting this inequality in \eqref{eq:lyap} gives
$$
\frac d{dt}\cL +c_0\cL^{1/\theta}\leq0\with c_0:=C^{-1/\theta}\cN_0^{1-1/\theta},$$
and thus, after time integration, 
\begin{equation}\label{eq:lyap2}
\cL(t)\leq \cL(0)\biggl(1+\frac{1-\theta}\theta\: c_0\cL_0^{\frac{1-\theta}{\theta}}t\biggr)^{-\frac\theta{1-\theta}}\cdotp\end{equation}
{In the present situation, $\cL$ will be either $E_0$ or $E_1,$ and $\cH$ will be either  $D_0$, or 
some  functional $D_1\simeq\|u_t\|_{L^2}^2+\||u-v|^2f\|_{L^1(\R^3_x\times\R^3_v)}^2$.  
For $\cN$ we will use the Besov norm $\|u\|_{\dot B^{-3/2}_{2,\infty}}$ rather than  the $L^1$ norm that becomes
badly for the Stokes system.  For technical reasons, we will not obtain  the optimal rates first time. 
We will first prove decay $t^{-3/2}$ for $E_1,$ which will enable us to bound 
the Brinkman force in $L^1(\R_+\times\R^3\times\R^3).$   }
Then, in a second step, we will obtain the optimal rates for both  $E_0$ and $E_1.$ 

Let us also underline that establishing \eqref{eq:interpo} is  not 
so straightforward  since $E_0,$ $D_0,$ $E_1$ and $D_1$ 
contain two terms that are hardly comparable.  Fortunately, being in the small solutions regime will 
spare us considering  different  cases, depending on whether one or the other term is dominant. 
 \smallbreak
In the third step, we show that under the smallness condition \eqref{eq:smalldata}, 
the a priori  assumptions that have been made hitherto are satisfied.  
{The key is  to establish that 
 \begin{equation}\label{eq:Lip}
	\int_0^T\|\nabla u\|_{L^\infty}\,dt\leq \delta,
	\end{equation}
	where the (universal) value of $\delta$ is given by \cite[Lemma 4.4]{HMM}. 
	It has been pointed out there that having \eqref{eq:Lip} allows us to bound the density in terms of the data. }
	
{Inequality \eqref{eq:Lip} will be obtained through a bootstrap argument
	by combining the time decay results of the second step
with the  energy estimates of the first step and suitable Gagliardo-Nirenberg or Sobolev inequalities. }
\smallbreak
The next step is uniqueness. To this end, we follow  the method of \cite{HMMM}
that relies on the control  of suitably weighted $L^2$ norms for  the difference of velocities and 
of the characteristics associated to the transport equation for the distribution function
(see the definition in \eqref{eq:car}). 
The situation here  is simpler since we have $\nabla u$ in $L^1(\R_+;L^\infty)$ and it is thus 
possible to conclude by means of the classical Gronwall lemma. 
Nevertheless, we give some details since this proof will be a model one to investigate the more complicated
situation of uniqueness for Fujita-Kato type solutions. It is also useful for
constructing global solutions.
\smallbreak
The final step  is to prove the  existence of a global solution. 
We use a Friedrichs regularization  of  the velocity equation
and of the velocity field  in the first equation of (VNS) (in order to keep  the symmetric structure of the overall system). 
These `regularized' solutions satisfy the same estimates as those of (VNS). 
Furthermore, by modifying suitably the proof of uniqueness, we can get the \emph{strong} convergence
to solutions of (VNS)  in the energy space. 
\medbreak
To get Theorem \ref{thm:main2},  the first step is to establish the short time
existence for $u_0$ in the Sobolev space $H^{1/2}.$  
As in \cite{DHK}, by looking at the velocity equation as a Navier-Stokes equation with 
a source term (namely the Brinkman force), we succeed in getting a solution on, say, the time interval $[0,1]$ provided
the data are small enough.  This solution is shown to have a log Lipschitz velocity (see \eqref{eq:loglip})
which enables us  to adapt  the method of \cite{HMMM} (based on Osgood lemma) 
to the 3D case so as to get uniqueness.
 Next, since the constructed solution is $H^1$ for almost every positive time and  as the negative Besov regularity is under control  for small time, even if $u$  is only a $H^{1/2}$ type solution, 
 one can combine this result with the previous
one to complete the proof of Theorem \ref{thm:main2}. 
\medbreak 
We end this section pointing out some interesting open questions related to our results. 
\begin{itemize}
\item  For (NS), assuming only that $u_0$ is small in $\dot H^{1/2}$ allows to get a global-in-time result 
(see \cite{FK} or the adaptation to the $\R^3$ setting  by Chemin in \cite{JYC}). 
Can we do without any additional integrability condition on $u_0$ in Theorems \ref{thm:main1} and  \ref{thm:main2}~?
\item  The localization requirement for $f_0$ is rather strong and not natural inasmuch as it does not come into play (at least not directly) in 
the control of the energy functionals $E_0$ and $E_1.$ Is it enough to make the same assumptions as in 
Theorem \ref{thm:main1}~? 
\item  Here we proved that  at large time,  $u$ tends to  the solution of the incompressible 
 Navier-Stokes equations with no source term, and that the density $\rho$ tends to be transported by its flow. 
 However, plugging the ansatz $f(t,x,v)=\rho(t,x)\,\delta_{v=w(t,x)}$ in (VNS) and observing that we thus have  $j=\rho w,$ we formally obtain the following \emph{pressureless  Euler-Navier-Stokes system} governing the evolution of $(\rho,w,u)$:
 $$\left\{\begin{array}{l}
\d_t\rho+\divx(\rho w)=0,\\[1ex]
\d_t(\rho w)+\divx(\rho w\otimes w)+\rho(w-u)=0,\\[1ex]
\d_t u+u\cdot\nabla_x  u-\Delta_x u+\nabla_x P=\rho(w-u),\\[1ex]
\div u=0.\end{array}\right.\leqno(ENS)$$
In our recent work \cite{D}, we proposed an approach to solve (ENS)  in a functional 
setting  which has some similarity with the one of Theorem \ref{thm:main1}.
Is there a framework in which we can compare the solutions offered by the two systems~?
 \end{itemize}

\medbreak
The rest of the paper unfolds as follows.  The next section is devoted 
to the proof of (conditional) energy estimates for smooth enough solutions of (VNS). 
Then, in Section \ref{s:decay}, we establish the optimal decay estimates for $E_0$ and $E_1,$ 
and deduce the key Lipschitz control of the velocity field. 
 Section \ref{s:thm1} is devoted to completing the proof of Theorem \ref{thm:main1}
 while Section \ref{s:FK} deals with the proof of Theorem \ref{thm:main2} and Remark \ref{r:main2}. 
Some useful results concerning the properties of the flow associated to the transport equation
of (VNS) are postponed in Appendix \ref{s:AppendixA}  and, {for the convenience of the reader, some basic facts about Besov spaces, real interpolation and parabolic maximal regularity are listed in  Appendix \ref{s:AppendixB}.}
\medbreak
\noindent{\bf Notation:} For any normed space $X,$   index $q\in[1,\infty]$ and time $T\in[0,\infty],$ 
we denote $\|z\|_{L^q_T(X)}:= \bigl\| \|z(t)\|_X\|_{L^q(0,T)}$  and omit $T$ if it is $\infty.$
In the case where $z$ has several components  in $X,$ we  keep the same notation for the norm. 
It will be sometimes convenient to use the short notation $u_t$  (resp. $f_t$) to designate the time derivative
of $u$ (resp. $f$).

  Finally, the notation $C$ stands for harmless positive real numbers, and we shall often write
$A\lesssim B$ instead of $A\leq CB.$ To emphasize the dependency with respect to parameters $a_1,\cdots,a_n,$
we adopt the notation $C_{a_1,\cdots,a_n}.$


\section{Energy  estimates} \label{s:energy}

Throughout this section, we assume that we are given a smooth solution $(f,u,P)$ of (VNS) on the time interval $[0,T].$
{We fix some $R\geq1$ and $M>0$ such that
\begin{equation}\label{eq:R}
\rho(t,x)\leq R\andf m_2(t,x)\leq M\quad\hbox{on}\quad [0,T]\times\R^3.\end{equation}
The main goal of this section is to derive a priori estimates involving the energy functionals
\begin{equation}\label{eq:E}
\begin{aligned}
E_{0}&:= \frac 12\|u\|_{L^2(\R^3_x)}^2+\frac12 \||v|^2 f\|_{L^1(\R^3_x\times\R^3_v)},\\
E_{1}&:=\|\nabla u\|_{L^2(\R^3_x)}^2+\||u-v|^2 f\|_{L^1(\R^3_x\times\R^3_v)},\\
E_{2}&:=\|u_t\|_{L^2(\R^3_x)}^2+\||u-v|^2 f\|_{L^1(\R^3_x\times\R^3_v)}\end{aligned}\end{equation}
and the dissipation rates:
\begin{equation}\label{eq:D}\begin{aligned}
D_{0}&:= \|\nabla u\|_{L^2(\R^3_x)}^2+\||u-v|^2 f\|_{L^1(\R^3_x\times\R^3_v)},\\
D_{1}&:=\frac12\|\nabla u\|_{L^2(\R^3_x)}^2+\frac12\||u-v|^2 f\|_{L^1(\R^3_x\times\R^3_v)}+
\frac{1}{24R}\|(\nabla^2u,\nabla^2 P)\|_{L^2(\R^3_x)}^2,\\
D_{2}&:=\|\nabla u_t\|_{L^2(\R^3_x)}^2+\|\sqrt \rho\, u_t\|_{L^2(\R^3_x)}^2\end{aligned}\end{equation}
in terms of norms of the initial data, of $R$ and of the Lipschitz norm of $u.$
The results of this section are summarised in the following proposition:
\begin{proposition}   Let $(f,u,P)$ be a smooth solution of  (VNS) on the time interval $[0,T]$
satisfying \eqref{eq:R} and decaying to $0$ at infinity. Then, we have the following results:
\begin{itemize}
\item the basic energy balance \eqref{eq:energy} holds true on $[0,T];$\smallbreak
\item  there exists a universal constant $C$ such that 
 \begin{equation}\label{eq:H1b}
\frac d{dt}E_1 +  D_1\leq C\|\nabla u\|_{L^\infty}E_1+CR^3\|\nabla u\|_{L^2}^6;
\end{equation}
\item  there exists a universal constant $C$ such that 
  \begin{multline}\label{eq:tE2} \frac{d}{dt}(tE_2) +tD_2 + \int\!\!\!\int  |v-u|^2f \,dv\,dx\leq 24RD_1+2\bigl(1+M\bigr)tD_1
 \\ +C\|\rho\|_{L^\infty}^2 D_0(tE_1)^2 + C\|\rho\|_{L^\infty}\sqrt{RD_0D_1}\: tE_1
 +\bigl(2\|\nabla u\|_{L^\infty}+C\|\rho\|_{L^\infty}\sqrt{RD_0D_1}\bigr)tE_2.  
  \end{multline}
\end{itemize}
\end{proposition}}
\begin{proof}
The energy balance can be obtained by  
taking the $L^2(\R^3;\R^3)$ scalar product of the velocity equation of (VNS) with $u,$ 
integrating the equation of $f$ on $\R^3\times\R^3$ with respect to the measure $|v|^2\,dx\,dv$
and performing suitable integration by parts (see e.g. \cite{BDGM}). 
\smallbreak
In order to establish \eqref{eq:H1b}, the starting point is to take the $L^2(\R^3;\R^3)$ scalar product of the velocity equation of (VNS) with $u_t.$ 
This gives
\begin{equation}\label{eq:H1}\frac12\frac d{dt}\|\nabla u\|_{L^2}^2+\|u_t\|_{L^2}^2=-\int (u\cdot\nabla u)\cdot u_t\,dx +
\int\!\!\!\int u_t\cdot((v-u)f)\,dv\,dx.\end{equation}
The convection term may be handled  as follows  for all $\eps>0$:
\begin{align}\label{eq:conv}\nonumber
-\int (u\cdot\nabla u)\cdot u_t\,dx&\leq \|u\|_{L^6}\|\nabla u\|_{L^3}\|u_t\|_{L^2}\\
&\leq C\|\nabla u\|_{L^2}^{3/2}\|\nabla^2u\|_{L^2}^{1/2} \|u_t\|_{L^2}\\
&\leq\frac14\|u_t\|_{L^2}^2+\frac\eps2\|\nabla^2u\|_{L^2}^2 +\frac C\eps\|\nabla u\|_{L^2}^6.\nonumber\end{align}
To bound  the last term of \eqref{eq:H1}, we use the decomposition
$$\begin{aligned}
\int\!\!\!\int u_t\cdot((v-u)f)\,dv\,dx
&=- \frac12 \frac d{dt}\int\!\!\!\int|v-u|^2f\,dv\,dx
+\frac12\int\!\!\!\int |v-u|^2f_t\,dv\,dx\\
&= -\frac12 \frac d{dt}\int\!\!\!\int|v-u|^2f\,dv\,dx
-\frac12\int\!\!\!\int |v-u|^2 v\cdot\nabla_x f\,dv\,dx\\
&\hspace{3cm}+\frac12\int\!\!\!\int |v-u|^2 \divv{((v-u)f)}\,dv\,dx\\
&= -\frac12 \frac d{dt}\int\!\!\!\int|v-u|^2fdv\,dx
+\int\!\!\!\int {((v-u)f)}\cdot(v\cdot\nabla u)dv\,dx\\
&\hspace{6cm}-\int\!\!\!\int  |v-u|^2f \,dv\,dx.
\end{aligned}$$
Hence we have
\begin{multline}\label{eq:vuf}
\int\!\!\!\int u_t\cdot((v-u)f)\,dv\,dx
= -\frac12 \frac d{dt}\int\!\!\!\int|v-u|^2f\,dv\,dx
+\int\!\!\!\int {((v-u)f)}\!\cdot\!(u\cdot\nabla u)\,dv\,dx\\+\int\!\!\!\int {((v-u)f)}\cdot((v-u)\cdot\nabla u)\,dv\,dx-\int\!\!\!\int  |v-u|^2f \,dv\,dx.
\end{multline}
The first and  last terms  are parts of $E_1$ and of $D_1$ (see  \eqref{eq:E} and \eqref{eq:D}), respectively. 
For the third term, we have 
$$\int\!\!\!\int {((v-u)f)}\cdot((v-u)\cdot\nabla u)\,dv\,dx
\leq \|\nabla u\|_{L^\infty} E_1.$$
Finally, observe that by Cauchy-Schwarz inequality with respect to the $v$ variable, we have  
\begin{equation}\label{eq:L2}
\Bigl\|\int {(v-u)f}\,dv\Bigr\|_{L^2}^2\leq\|\rho\|_{L^\infty}\int\!\!\!\int f|v-u|^2\,dx\,dv.
\end{equation}
Hence
$$\begin{aligned}
\int\!\!\!\int  {((v-u)f)}\cdot(u\cdot\nabla u)\,dv\,dx&\leq \Big\|\int {(v-u)f}\,dv\Big\|_{L^2} \|u\|_{L^6}\|\nabla u\|_{L^3}\\
&\leq C\|\rho\|_{L^\infty}^{1/2}\biggl(\int\!\!\!\int |v-u|^2f\,dv\,dx\biggr)^{1/2}\|\nabla u\|_{L^2}^{3/2}\|\nabla^2 u\|_{L^2}^{1/2}\\
&\leq \frac\eps4\|\nabla^2u\|_{L^2}^2\!+\!\frac12\int\!\!\!\int \! |v-u|^2fdv\,dx+\frac C\eps\|\rho\|_{L^\infty}^2\|\nabla u\|_{L^2}^6.
\end{aligned}$$
Using the definition of $E_1$ and  \eqref{eq:R},   we conclude  that for all $\eps>0,$
\begin{multline}\label{eq:H10b}
\frac d{dt}E_1\!+\!\|u_t\|_{L^2}^2\!+\!\int\!\!\!\int |u-v|^2f\,dv\,dx\leq \eps\|\nabla^2u\|_{L^2}^2 + C\|\nabla u\|_{L^\infty}E_1+
C\eps^{-1}R^2\|\nabla u\|_{L^2}^6.
\end{multline}
To close the estimate, one has to bound $\|\nabla^2u\|_{L^2}.$ This may be done by using the 
elliptic regularity of the following Stokes system:
\begin{equation}\label{eq:stokes}-\Delta u+\nabla P=-u_t-u\cdot \nabla u+\int {(v-u)f}\,dv,\qquad \div u=0.\end{equation}
We have
$$\|\nabla^2u,\nabla P\|_{L^2}^2\leq 3\|u_t\|_{L^2}^2+ 3\|u\cdot\nabla u\|_{L^2}^2+3\Bigl\|\int {(v-u)f}\,dv\Bigr\|_{L^2}^2.$$
So, handling the convection term  as in \eqref{eq:conv}  and using \eqref{eq:L2}, we discover that
$$\|\nabla^2u,\nabla P\|_{L^2}^2\leq 3\|u_t\|_{L^2}^2+C\|\nabla u\|_{L^2}^3\|\nabla^2 u\|_{L^2}
+3 \|\rho\|_{L^\infty}\int\!\!\!\int |v-u|^2f\,dx\,dv,$$
which, leveraging Young inequality, leads to 
$$\|\nabla^2u,\nabla P\|_{L^2}^2\leq 6\|u_t\|_{L^2}^2+C\|\nabla u\|_{L^2}^6
+6 \|\rho\|_{L^\infty}\int\!\!\!\int |v-u|^2f\,dx\,dv.$$
{Keeping \eqref{eq:R} in mind,
and reverting  to \eqref{eq:H10b} with $\eps=R^{-1},$ we end up with  \eqref{eq:H1b}.}

\smallbreak
In order to establish \eqref{eq:tE2},  let us  differentiate the velocity equation with respect to time: 
\begin{align}u_{tt} +u\cdot\nabla u_t +\nabla P_t -\Delta u_t +\rho u_t + u_t\cdot\nabla u&=-\int(u-v)f_t\,dv\nonumber\\\label{eq:utt}
&= \int (u-v)(v\cdot\nabla_x f)\,dv +\int (v-u)f\,dv.\end{align}
 Taking the $L^2(\R^3;\R^3)$ scalar product with $t u_t$ yields
 \begin{multline*}\frac12\frac{d}{dt}\|\sqrt t\, u_t\|_{L^2}^2 + \|\sqrt t\, \nabla u_t\|_{L^2}^2 +\|\sqrt{\rho t}\, u_t\|_{L^2}^2 -
 \frac12\|u_t\|_{L^2}^2 \\ =- \int (tu_t)\cdot (u_t\cdot\nabla u)\,dx +\int\!\!\!\int (v\cdot\nabla_xf) (u-v)\cdot (tu_t)\,dv\,dx
 +\int\!\!\!\int (tu_t)\cdot (v-u)f\,dv\,dx.
 \end{multline*}
 In light of \eqref{eq:vuf}, we have 
 \begin{multline*}\int\!\!\!\int (tu_t)\cdot (v-u)f\,dv\,dx=
 -\frac12 \frac d{dt}\int\!\!\!\int  t|v-u|^2f\,dv\,dx
+\int\!\!\!\int t  {((v-u)f})\cdot\!(u\cdot\nabla u)\,dv\,dx\\+\int\!\!\!\int t  {(v-u)f} \cdot((v-u)\cdot\nabla u)\,dv\,dx
-\frac12\int\!\!\!\int f |v-u|^2 \,dv\,dx.\end{multline*}
 Therefore, setting 
\begin{equation}\label{eq:E2} E_2:=\|u_t\|_{L^2}^2 + \int\!\!\!\int f |v-u|^2 \,dv\,dx
   \andf D_2:= \| \nabla u_t\|_{L^2}^2 +\|\sqrt{\rho}\, u_t\|_{L^2}^2,\end{equation}
   we discover that 
    \begin{multline}\label{eq:E2t} \frac12\frac{d}{dt}(tE_2) +tD_2 +\frac12 \int\!\!\!\int f |v-u|^2 \,dv\,dx
    =  \frac12\|u_t\|_{L^2}^2 - \int (tu_t)\cdot (u_t\cdot\nabla u)\,dx \\+\int\!\!\!\int (v\cdot\nabla_xf) (u-v)\cdot (tu_t)\,dv\,dx
    \\+\int\!\!\!\int t {((v-u)f)}\cdot(u\cdot\nabla u)\,dv\,dx+\int\!\!\!\int t  {((v-u)f)}\cdot((v-u)\cdot\nabla u)\,dv\,dx.
    \end{multline}
 For the second term of the right-hand side, combining H\"older, Sobolev, Gagliardo-Nirenberg
 (to bound the $L^3$ norm) 
  and, eventually, Young inequalities gives
 $$ \begin{aligned} - \int (tu_t)\cdot (u_t\cdot\nabla u)\,dx&\leq \|\sqrt t\, u_t\|_{L^6}\|\sqrt t\, u_t\|_{L^2}\|\nabla u\|_{L^3}\\
 &\leq\frac14 \|\sqrt t\, \nabla u_t\|_{L^2}^2 
 + C \|\sqrt t\, u_t\|_{L^2}^2\|\nabla u\|_{L^2}\|\nabla^2 u\|_{L^2}.\end{aligned}$$
 To handle the third term, we integrate by parts and use that $v= u+(v-u),$ getting~:
 \begin{multline*} 
\int\!\!\!\int (v\cdot\nabla_xf) (u-v)\cdot (tu_t)\,dv\,dx 
=I_1+I_2+I_3:=\int\!\!\!\int t(v-u) v\cdot\nabla u_t\,f\,dv\,dx \\- \int\!\!\!\int tu_t\cdot(u\cdot\nabla u)\, f\,dv\,dx
+ \int\!\!\!\int tu_t\cdot((u-v)\cdot\nabla u)\, f\,dv\,dx.
 \end{multline*} 
 Combining Cauchy-Schwarz and Young inequalities, we easily get for all $\eps>0,$
 $$I_1\leq\frac\eps4 \|\sqrt{tm_2}\, \nabla u_t\|_{L^2}^2 + \frac1\eps \int\!\!\!\int tf|v-u|^2\,dv\,dx. $$
Next, we have 
$$\begin{aligned}
I_2&\leq \|\rho\|_{L^\infty}^{1/2} \|\sqrt{\rho t}\, u_t\|_{L^2} 
\|\nabla u\|_{L^3}\|\sqrt t u\|_{L^6}\\
&\leq \frac12 \|\sqrt{\rho t}\,u_t\|_{L^2}^2+Ct\|\rho\|_{L^\infty}\|\nabla u\|_{L^2}^3\|\nabla^2u\|_{L^2}\\
&\leq \frac12 \|\sqrt{\rho t}\,u_t\|_{L^2}^2+\frac12\|\nabla^2 u\|_{L^2}^2
+C\|\rho\|_{L^\infty}^2\|\nabla u\|_{L^2}^2\|\sqrt t\nabla u\|_{L^2}^4
\end{aligned}$$
and,  using  \eqref{eq:L2} and Gagliardo-Nirenberg inequality to bound the $L^3$ norm, 
$$\begin{aligned}
I_3&\leq \|\sqrt t u_t\|_{L^6} \|\nabla u\|_{L^3}\Bigl\|\sqrt t\int(v-u)f\Bigr\|_{L^2}\\
&\leq \frac14 \|\sqrt t\, \nabla u_t\|_{L^2}^2 + C\|\rho\|_{L^\infty}\|\nabla u\|_{L^2}\|\nabla^2u\|_{L^2}\int\!\!\!\int tf|v-u|^2\,dv\,dx. 
\end{aligned}
$$
Next, we write that
$$
\begin{aligned}\int\!\!\!\int \!t {((v-u)f)}\cdot(u\cdot\nabla u)\,dvdx&\leq \|\rho\|_{L^\infty}^{1/2}\biggl(\int\!\!\!\int\!t|v-u|^2f\,dvdx\biggr)^{\!\!1/2}
\|\sqrt t\,u\|_{L^6}\|\nabla u\|_{L^3}\\
\leq&\, C \|\rho\|_{L^\infty}^{1/2}\biggl(\int\!\!\!\int \!t|v-u|^2f\,dvdx\biggr)^{\!\!1/2}
\|\sqrt t\,\nabla u\|_{L^2}\|\nabla u\|_{L^2}^{1/2}\|\nabla^2u\|_{L^2}^{1/2}\\
\leq&\, \frac12\int\!\!\!\int\! t|v-u|^2f\,dvdx+C \|\rho\|_{L^\infty}\|\nabla u\|_{L^2}\|\nabla^2u\|_{L^2}
\|\sqrt t\,\nabla u\|_{L^2}^2.\end{aligned}$$
Finally,  it is obvious that
$$\int\!\!\!\int t {((v-u)f)}\cdot((v-u)\cdot\nabla u)\,dv\,dx\leq \|\nabla u\|_{L^\infty}\int\!\!\!\int t|v-u|^2f\,dv\,dx.$$
Therefore,  choosing $\eps =M^{-1},$ reverting to \eqref{eq:E2t} and using the definition of $D_i$ and $E_i$ for $i=0,1,2,$ 
we end up with \eqref{eq:tE2}.
  \end{proof}
 {The final goal of this section is to get a control on the energy functionals $E_0,$ $E_1$ and $tE_2$ on $[0,T]$ in terms of the 
  data, assuming in addition that the data are small enough, and that the Lipschitz condition \eqref{eq:Lip} is satisfied. 
  This is summarized in the following corollary:
  \begin{corollary}\label{c:final} Let  $(f,u,P)$ be a smooth solution of $(VNS)$ on the time interval $[0,T]$ satisfying the Lipschitz smallness condition \eqref{eq:Lip}. 
  Set $$\mathfrak f_0:=\|\langle v\rangle^2 f_0\|_{L^1_v(\R^3;L^\infty_x(\R^3))}
  \andf R_0:=\max(1,2\|f_0\|_{L^1_v(\R^3;L^\infty_x(\R^3))}).$$  
    There exist two   absolute constants $c_0$  and $C$  such that, if
 \begin{equation}\label{eq:smalldata1}
 \|u_0\|_{H^1}^2 + \int\!\!\!\int |v|^2 f_0\,dv\,dx\leq 
 \frac{c_0}{R_0^2}
 \end{equation}
  then, we have for all $t\in[0,T],$ 
\begin{equation}\label{eq:cE}
\cE(t)+\int_0^t\cD\,d\tau\leq C\bigl(1+\mathfrak f_0\bigr)\biggl(\|u_0\|_{H^1}^2 
+ \int\!\!\!\int |v|^2 f_0\,dv\,dx\biggr)
\end{equation}  
with, for all $t\in[0,T],$ 
\begin{equation}\label{eq:DE} \begin{aligned}\cE(t)&:=2(2+C\mathfrak f_0)(tE_1(t)+2E_0(t)) +25 R_0 E_1(t)+tE_2(t)\\
	\andf \cD(t)&:= 2(1+C\mathfrak f_0)D_0(t)+2 tD_1(t) +R_0D_1(t)+tD_2(t).\end{aligned}\end{equation}
 \end{corollary}
  \begin{proof}  For the time being,  take $M>0$ and $R\geq1$ such that \eqref{eq:R} holds on $[0,T],$ and  set 
	\begin{equation*} \begin{aligned}\cE(t)&:=2(2+M)(tE_1(t)+2E_0(t)) +25 R E_1(t)+tE_2(t)\\
	\andf \cD(t)&:= 2(1+M)D_0(t)+2 tD_1(t) +RD_1(t)+tD_2(t).\end{aligned}\end{equation*}
	Using a suitable linear combination  of   \eqref{eq:energy}, \eqref{eq:H1b} and \eqref{eq:tE2}, and observing that
	$$	\frac d{dt}(tE_1)=t\frac d{dt} E_1+D_0,$$ 
 we discover that
	\begin{equation}\label{eq:cEbis} \frac d{dt}\cE+\cD\leq C\bigl(U\cE+\cD \cE +R^2D_0\cE^2\bigr)\with U:=\|\nabla u\|_{L^\infty}.
	\end{equation}
	Consequently, as long as 
	\begin{equation}\label{eq:bootstrap}
	CR^2\cE\leq1/2,\end{equation} we  have 
	\begin{equation}\label{eq:ED} \frac d{dt}\cE+\frac12\cD\leq C\bigl(U +D_0\bigr)\cE\end{equation}
	and  Gronwall lemma  thus gives us, after a harmless change of $C,$
	\begin{equation}\label{eq:cE1}\cE(t)+\int_0^t\cD\,d\tau\leq C(R+M)(E_{0,0}+E_{1,0})e^{C\int_0^t (U+D_0)\,d\tau}.\end{equation}
Since  \eqref{eq:Lip} is satisfied,  Inequalities  \eqref{eq:boundrho} and  \eqref{eq:bounde} hold true, and one can thus choose $R=R_0$
 and $M=C\mathfrak f_0.$
Using also \eqref{eq:energyb},  Inequality \eqref{eq:cE1} finally implies that as long as \eqref{eq:bootstrap} holds, we have
\begin{equation}\label{eq:cE2}\cE(t)+\int_0^t\cD\,d\tau\leq C  \bigl(1+\mathfrak f_0\bigr)
e^{CE_{0,0}}.\end{equation}
Leveraging a bootstrap argument, we conclude that \eqref{eq:bootstrap} is satisfied if
$R_0^2\bigl(E_{0,0}+E_{1,0}\bigr)e^{CE_{0,0}}$ is small enough.
Since we have $e^{CE_{0,0}}\leq2$ if $E_{0,0}$ is sufficiently small and, clearly, 
$$E_{0,0}+E_{1,0}\simeq \|u_0\|_{H^1}^2 + \int\!\!\!\int |v|^2 f_0\,dv\,dx,$$
this gives   the desired statement. 
\end{proof}}


\section{Decay estimates and control of the Lipschitz norm of the velocity}\label{s:decay}

The first goal of the present section is to establish optimal time decay for the energy functionals $E_0$ and $E_1$
associated to  global  solutions to (VNS), \emph{under the Lipschitz condition \eqref{eq:Lip}}. 
Secondly, we will combine these estimates with those from the previous section and interpolation arguments to exhibit a sufficient condition on the initial data to ensure \eqref{eq:Lip}.

\subsection{Decay estimates for the energy functionals} 

The main goal of this part is to establish the following statement: 
{\begin{proposition} Let $(f,u,P)$ be a smooth solution of (VNS) on $[0,T],$ satisfying \eqref{eq:Lip}. 
There exist two positive real numbers $a_0$ and $a_1$ depending only on  $M_0,$ $E_{0,0},$ $E_{1,0},$ $\|u_0\|_{\dot B^{-\frac32}_{2,\infty}}$
 and $R$  such that
if $E_{0,0}$ and $E_{1,0}$ are small enough then,  we have for all $t\in[0,T]$:
\begin{equation}\label{eq:EEE}E_0(t)\leq CE_{0,0}(1+a_0t)^{-3/2}\andf  E_1(t)\leq CE_{1,0} (1+a_1t)^{-5/2}.\end{equation}
Furthermore, we have for all $\alpha\in (0,3/2)$:
\begin{align}\label{eq:DDD0}
\bigl(\textstyle{\frac 32}-\alpha\bigr)\Int_0^t(1+a_0\tau)^{\alpha}D_0(\tau)\,d\tau&\leq CE_{0,0}\\\andf\label{eq:DDD1}
\bigl(\textstyle{\frac 32}-\alpha\bigr)\Int_0^t(1+a_1\tau)^{\alpha+1}D_1(\tau)\,d\tau&\leq C(E_{0,0}+E_{1,0}).
\end{align}\end{proposition}}
\begin{proof}
The proof is based  on Inequality \eqref{eq:lyap2} with 
$\cL$ (resp. $\cD$) being either $E_0$ or $E_1$ (resp. $D_0$ or $D_1$), and $\cN=\|u\|_{\dot B^{-3/2}_{2,\infty}}.$ 
The heuristics for $E_1$ relies on the fact that, roughly, for small solutions, 
we have $E_1\simeq \|\nabla u\|_{L^2}^2$ and $D_1\simeq \|\nabla^2 u\|_{L^2}^2,$ and that 
 the following interpolation inequality (that stems from  \eqref{eq:equivHs}, \eqref{eq:equivk} and \eqref{eq:interpo3}) holds true for all $\sigma>-1$:
\begin{equation}\label{eq:interpo2}
\|\nabla u\|_{L^2}\lesssim \|\nabla^2u\|_{L^2}^{\theta}\|u\|_{\dot B^{-\sigma}_{2,\infty}}^{1-\theta},
\with\theta=\frac{\sigma+1}{\sigma+2}\cdotp
\end{equation}
Provided it can be shown that we indeed have some control on $\|u\|_{\dot B^{-3/2}_{2,\infty}},$
applying Inequality \eqref{eq:lyap2} will give exactly the right decay exponent for $E_1.$
A similar heuristics applies to  $E_0.$ 
\smallbreak
Proving time-independent estimates for $\|u\|_{\dot B^{-3/2}_{2,\infty}}$  revolves on maximal regularity estimates for 
the heat equation, once having noticed that, denoting by $\cP$ the Leray projector on divergence-free vector-fields, we have
\begin{equation}\label{eq:mild}\partial_tu-\Delta u= -\cP\div(u\cdot\nabla u) + \cP\biggl(\Int  {(v-u)f}\,dv\biggr)\cdotp\end{equation}
Indeed, {owing to \eqref{eq:cP} and  Proposition \ref{p:maxreg}, we have}
\begin{equation}\label{eq:3/2}\|u\|_{L^\infty_t(\dot B^{-\frac32}_{2,\infty})} + \|u\|_{\wt L_t^1({\dot B^{\frac12}_{2,\infty}})}\leq  \|u_0\|_{\dot B^{-\frac32}_{2,\infty}}
+\|u\cdot\nabla u\|_{\wt L_t^1(\dot B^{-\frac32}_{2,\infty})} 
+ \biggl\|\Int  {(u-v)f}\,dv\biggr\|_{\wt L_t^1(\dot B^{-\frac32}_{2,\infty})}.\end{equation}
Combining with the chain of embeddings (see  \eqref{eq:embed} and \eqref{eq:Minko}):
\begin{equation}\label{eq:chain}L^1(0,t\times\R^3)\hookrightarrow L^1(0,t;\dot B^{-\frac32}_{2,\infty}(\R^3))
\hookrightarrow \wt L^1_t(\dot B^{-\frac32}_{2,\infty}(\R^3)),\end{equation}
we thus have 
\begin{equation}\label{eq:3/2b}\|u\|_{L^\infty_t(\dot B^{-\frac32}_{2,\infty})} + \|u\|_{\wt L_t^1({\dot B^{\frac12}_{2,\infty}})}\leq  \|u_0\|_{\dot B^{-\frac32}_{2,\infty}}
+\|u\cdot\nabla u\|_{L^1(0,t\times\R^3)}+\biggl\|\Int  {(u-v)f}\,dv\biggr\|_{L^1(0,t\times\R^3)}.\end{equation}
Having time-independent control of the left-hand side is fundamental in Nash' argument. 
However,  at this stage, it is not clear how to control the Brinkman term in $L^1(\R_+\times\R^3).$ 
The only information that we one can get relatively easily is a control in $L^2(\R_+;L^1(\R^3)).$
Using again   Proposition \ref{p:maxreg} and the above chain embedding (with $L^2$-in-time instead of $L^1$), 
this will provide us with a control on the $\dot B^{-1/2}_{2,\infty}$ norm and then, with decay $t^{-3/2}$ for $E_1.$
This decay  turns  out to be fast  enough to  control the Brinkman term in $L^1(\R_+\times\R^3),$ and to go back to the propagation 
of the   regularity $\dot B^{-3/2}_{2,\infty}.$
\smallbreak
{This motivates the following structure of the proof:
\begin{itemize}
\item[Step 1.] Propagating regularity $\dot B^{-1/2}_{2,\infty}.$
\item[Step 2.] Proving decay $t^{-3/2}$ for $E_1.$
\item[Step 3.] Propagating  regularity $\dot B^{-3/2}_{2,\infty}.$
\item[Step 4.] Proving optimal decay for $E_0$ and $E_1.$
\end{itemize}}


\subsubsection*{Step 1: Propagation of regularity $\dot B^{-1/2}_{2,\infty}$}

Applying Proposition \ref{p:maxreg} with $s=-1/2,$ $\rho_2=2,$ $p=2,$ $r=\infty$ and $\rho_1\in\{2,\infty\}$  to \eqref{eq:mild} gives  for all $t\in[0,T],$ 
$$\|u(t)\|_{\dot B^{-\frac12}_{2,\infty}} + \|u\|_{\wt L_t^2({\dot B^{\frac12}_{2,\infty}})}\leq  \|u_0\|_{\dot B^{-\frac12}_{2,\infty}}
+\|u\cdot\nabla u\|_{\wt L_t^2(\dot B^{-\frac32}_{2,\infty})}+\biggl\|\int {(v-u)f}\,dv\biggr\|_{\wt L_t^2(\dot B^{-\frac32}_{2,\infty})}.$$
Hence, combining with  the following embeddings that follow from \eqref{eq:embed} and \eqref{eq:Minko}, 
\begin{equation}\label{eq:tilde}L^2(0,t;L^1(\R^3))\hookrightarrow L^2(0,t;\dot B^{-\frac32}_{2,\infty}(\R^3))
\hookrightarrow \wt L^2_t(\dot B^{-\frac32}_{2,\infty}(\R^3)),\end{equation}
  we have 
  $$\|u(t)\|_{\dot B^{-\frac12}_{2,\infty}} + \|u\|_{\wt L_t^2({\dot B^{\frac12}_{2,\infty}})}\leq  \|u_0\|_{\dot B^{-\frac12}_{2,\infty}}
  +C\|u\cdot\nabla u\|_{L^2_t(L^1)}+C \biggl\|\Int  {(v-u)f}\,dv\biggr\|_{L^2_t(L^1)}\cdotp
  $$
  The last term may be bounded by means of Cauchy-Schwarz inequality as follows:
\begin{equation}\label{eq:L2L1}
\int\biggl|\int  {(v-u)f}\,dv\biggl|dx\leq \biggl(\int\!\!\!\int f\,dv\,dx\biggr)^{1/2}
\biggl(\int\!\!\!\int |v-u|^2f\,dv\,dx\biggr)^{1/2}.\end{equation}
Hence,  using  the relations  \eqref{eq:energy0} and \eqref{eq:M0},  we discover that
$$\biggl\|\Int {(v-u)f}\,dv\biggr\|_{L_t^2(L^1)}\leq C\sqrt{M_0E_{0,0}}\,.$$
Furthermore,  \eqref{eq:tilde}, Cauchy-Schwarz  inequality and \eqref{eq:energy0}   yield 
$$\|u\cdot\nabla u\|_{\wt L_t^2(\dot B^{-\frac32}_{2,\infty})}\lesssim
\|u\cdot\nabla u\|_{L^2_t(L^1)}\leq \|u\|_{L^\infty_t(L^2)}\|\nabla u\|_{L^2_t(L^2)}\leq \sqrt 2\,E_{0,0}.$$
In conclusion, we  have for all $t\in[0,T],$
\begin{equation}\label{eq:negative}
\|u\|_{L^\infty_t(\dot B^{-\frac12}_{2,\infty})} + \|u\|_{\wt L^2_t({\dot B^{\frac12}_{2,\infty}})}\leq  \|u_0\|_{\dot B^{-\frac12}_{2,\infty}}
+ C\sqrt{M_0E_{0,0}} + CE_{0,0}=:C_0.\end{equation}

\subsubsection*{Step 2: Proving decay $t^{-3/2}$ for $E_1$}

From \eqref{eq:interpo2} and \eqref{eq:negative}, we infer that 
$$ \|\nabla^2 u\|_{L^2}^2\gtrsim C_0^{-4/3}\bigl(\|\nabla u\|_{L^2}^2\bigr)^{5/3}.$$
Therefore,  if we set $$U:=\|\nabla u\|_{L^\infty},\quad\!\!
X_1:=\|\nabla u\|_{L^2}^2\andf Y_1:=\int\!\!\!\int |u-v|^2f\,dx\,dv,$$  and use  \eqref{eq:H1b} with $R=\max(1,2\|f_0\|_{L^1_v(\R^3;L^\infty_x(\R^3))})$, we end up 
with 
$$ \frac d{dt} E_1+2c'_0X_1^{5/3} + Y_1\leq CU E_1 +CR^3X_1^3\with c'_0:=C_0^{-4/3}/8.$$ 
So,  assuming the following a priori bound:
\begin{equation}\label{eq:apriori}
 CR^3X_1^{4/3}\leq c'_0, \end{equation}
we get
\begin{equation}\label{eq:wtE_1}
\frac d{dt}\wt E_1 + c'_0 \wt X_1^{5/3} + \wt Y_1\leq 0\with 
\wt Z(t):=e^{-C\int_0^t U(\tau)\,d\tau} Z(t)\quad\hbox{for }\ Z\in\{E_1,X_1,Y_1\}.\end{equation}
Note that 
$$\wt Y_1\leq \wt E_1\leq E_{1,0}.$$
Since we restricted our attention to the case of small data (meaning in particular
that $E_{1,0}$ is small), one may assume with no loss of generality that $\wt Y_1\geq c'_0 \wt Y_1^{5/3},$
and thus Inequality \eqref{eq:wtE_1} gives after a suitable harmless change of $c'_0$:
$$\frac d{dt}\wt E_1 + c'_0 \wt E_1^{5/3} \leq 0.$$
Hence, provided  \eqref{eq:Lip} holds, we have 
 \begin{equation}\label{eq:E1c} E_1(t)\leq  CE_{1,0} (1+a_1t)^{-3/2}\end{equation} with  $a_1$ 
depending only on the initial data, on  $M_0,$ $E_{0,0}$  and $\|u_0\|_{\dot B^{-1/2}_{2,\infty}}.$ 
\smallbreak
{Combining the above inequality with   \eqref{eq:H1b}  will enable us to exhibit some extra integrability for $D_1$. }
Indeed,  assuming that 
\begin{equation}\label{eq:smallness1}
CR^3E_1\leq1,\end{equation}
Inequality   \eqref{eq:H1b}  implies  that
\begin{equation}\label{eq:E1b}
 \frac d{dt}E_1 + D_1\leq C(U+D_0) E_1. \end{equation}
 Let us set  for some $\beta\in(0,3/2),$ 
$$\check E_1(t):=(1+a_1t)^{\beta}  e^{-C\int_0^t (U+D_0)\,d\tau}  E_1(t)\andf \check D_1(t):=(1+a_1t)^{\beta}  e^{-C\int_0^t (U+D_0)\,d\tau}  D_1(t).$$
Then \eqref{eq:E1b} implies that 
$$\frac d{dt}\check E_1+ \check D_1 \leq \beta a_1(1+a_1t)^{\beta-1}\wt E_1.$$
Hence, using the bound \eqref{eq:E1c} and observing that $t\mapsto (1+a_1t)^{\beta-5/2}$ is integrable on $\R_+,$ 
we end up with\footnote{Note that, if taking $\beta=3/2,$ then one gets
$\Int_0^t  (1\!+\!a_1\tau)^{3/2}  \wt D_1(\tau)\,d\tau \leq 2\bigl(1+\log(1\!+\!a_1t)\bigr){E_{1,0}}.$}
$$ \int_0^t  (1+a_1\tau)^\beta  \wt D_1(\tau)\,d\tau \leq \frac{2\beta}{3/2\,-\,\beta}\: E_{1,0}{\with 
\wt D_1(t):=e^{-C\int_0^t U(\tau)\,d\tau} D_1(t).}$$
 Therefore, assuming \eqref{eq:Lip} and \eqref{eq:smallness1} gives 
  \begin{equation}\label{eq:E1d}
 (3/2-\beta) \int_0^t  (1+a_1\tau)^\beta   D_1(\tau)\,d\tau \leq  C \,e^{CE_{0,0}} \:E_{1,0},\qquad \beta\in(0,3/2).
 \end{equation}

\subsubsection*{Step 3. Propagating  regularity $\dot B^{-3/2}_{2,\infty}$}

Leveraging  \eqref{eq:E1d} with e.g.  $\beta=5/4$ ensures that the Brinkman force belongs to $L^1(\R_+\times\R^3).$ 
Indeed, remembering \eqref{eq:L2L1} and using Cauchy-Schwarz inequality, one can write that 
\begin{equation}\label{eq:L1}
\int_{\R_+} \int \biggl|\int  {(u-v)f}\,dv\,\biggl| dxdt \leq  \sqrt{M_0} \int_{\R_+} \!\sqrt{ (1+a_1t)^{5/4}\, D_1(t)}\: (1+a_1t)^{-5/8}\,dt \leq C_{0}.\end{equation}
This additional information will enable us to control  $\|u(t)\|_{\dot B^{-\frac32}_{2,\infty}}$
through Inequality \eqref{eq:3/2b}.
Indeed, as a first, we observe that 
$$\|u\cdot\nabla u\|_{L^1_t(L^1)}\leq \|u\|_{L^2_t(L^2)}\|\nabla u\|_{L^2_t(L^2)}$$
and that (see just below \eqref{eq:interpo3}):
\begin{equation}\label{eq:interpo4}\|u\|_{L^2_t(L^2)}\lesssim \|u\|_{\wt L^2_t(\dot B^{-\frac12}_{2,\infty})}^{2/3}\|\nabla u\|_{L^2_t(L^2)}^{1/3}
\andf \|u\|_{\wt L^2_t(\dot B^{-\frac12}_{2,\infty})}\leq  \|u\|_{L^\infty_t(\dot B^{-\frac32}_{2,\infty})}^{1/2} \|u\|_{\wt L^1_t(\dot B^{\frac12}_{2,\infty})}^{1/2}.
\end{equation}
Hence, 
$$\|u\cdot\nabla u\|_{L^1_t(L^1)}\lesssim \|\nabla u\|_{L^2_t(L^2)}^{4/3}
  \|u\|_{L^\infty_t(\dot B^{-\frac32}_{2,\infty})}^{1/3} \|u\|_{\wt L^1_t(\dot B^{\frac12}_{2,\infty})}^{1/3}.$$
  Hence, owing to \eqref{eq:energy} and \eqref{eq:L1}, we have
  $$\|u\|_{L^\infty_t(\dot B^{-\frac32}_{2,\infty})\cap \wt L_t^1({\dot B^{\frac12}_{2,\infty}})}\leq  \|u_0\|_{\dot B^{-\frac32}_{2,\infty}}
+CE_{0,0}^{2/3}\|u\|_{L^\infty_t(\dot B^{-\frac32}_{2,\infty})\cap \wt L_t^1({\dot B^{\frac12}_{2,\infty}})}^{2/3}+C_0.$$  
Consequently, using Young inequality and \eqref{eq:energy}, 
one can conclude that there exists $C'_0$ depending only on $E_{0,0},$  $M_0$ and  $\|u_0\|_{\dot B^{-\frac32}_{2,\infty}}$
such that if  $E_{1,0}$ is small enough 
and \eqref{eq:Lip} is satisfied, then we have
\begin{equation}\label{eq:negative1}
\|u\|_{L^\infty_t(\dot B^{-\frac32}_{2,\infty})} + \|u\|_{\wt L^1_t({\dot B^{\frac12}_{2,\infty}})}\leq  C'_0,\qquad t\in[0,T].\end{equation}

\subsubsection*{Step 4. Optimal decay for $E_0$ and $E_1$}

From \eqref{eq:equivHs}, \eqref{eq:equivk} and \eqref{eq:interpo3}, we deduce that
\begin{equation}\label{eq:interpopo}\|u\|_{L^2}\lesssim \|\nabla u\|_{L^2}^{3/5} \|u\|_{\dot B^{-\frac32}_{2,\infty}}^{2/5}.\end{equation}
Inequality  \eqref{eq:negative1} thus ensures  that for some $C_0$ depending only on  $E_{0,0},$  $M_0$ and  $\|u_0\|_{\dot B^{-\frac32}_{2,\infty}},$
$$\|\nabla u\|_{L^2}^2\geq C_0\bigl(\|u\|_{L^2}^2\bigr)^{5/3}.$$
At this stage, one can almost mimic the proof of \eqref{eq:E1c} and \eqref{eq:E1d}.  {However, one has to pay attention 
to the fact that, in contrast with $E_1$ and $D_1,$ the kinetic parts of $E_0$ and $D_0$ are not equal.
Now, in light of H\"older and Young inequalities,  and  embedding $\dot H^1(\R^3)\hookrightarrow L^6(\R^3),$ we get
$$\begin{aligned}
\int\!\!\!\int f|u-v|^2\,dx\,dv&=\int\!\!\!\int f|u|^2\,dx\,dv+\int\!\!\!\int f|v|^2\,dx\,dv -2\int\!\!\!\int f u\cdot v\,dx\,dv\\
&\geq \frac12\int\!\!\!\int f|v|^2\,dx\,dv-\int\rho|u|^2\,dx\\
&\geq \frac12\int\!\!\!\int f|v|^2\,dx\,dv-\|\rho\|_{L^{3/2}}\|u\|_{L^6}^2\\
&\geq \frac12\int\!\!\!\int f|v|^2\,dx\,dv-CM_0^{2/3}R^{1/3}\|\nabla u\|_{L^2}^2.
\end{aligned}$$
Combining with the obvious inequality (valid for all $\eps\in[0,1]$): 
$$\int\!\!\!\int f|u-v|^2\,dx\,dv\geq \eps\int\!\!\!\int f|u-v|^2\,dx\,dv\geq \frac\eps2\int\!\!\!\int f|v|^2\,dx\,dv
-C\frac\eps2 M_0^{2/3}R^{1/3}\|\nabla u\|_{L^2}^2$$
and taking $\eps=\min(1,1/(CM_0^{2/3}R^{1/3})),$ we get 
$$D_0\geq \frac12 \|\nabla u\|_{L^2}^2 + \frac\eps2\||v|^2f\|_{L^1(\R^3_x\times\R^3_v)}.$$
Arguing as in Step 2, it is now easy to conclude to the second part of \eqref{eq:EEE} and to \eqref{eq:DDD0}.}
\smallbreak
Finally, in order to improve the decay of $E_1$ and $D_1 $ by an increment $1,$  it suffices to use  \eqref{eq:interpo2}
with $\sigma=3/2$ 
which, owing to \eqref{eq:negative1}, gives 
$$\|\nabla^2 u\|_{L^2}^2\gtrsim C_0^{-4/5} \bigl(\|\nabla u\|_{L^2}^2\bigr)^{7/5}.$$
Hence, we have for a new constant $c_0$ depending only on $\|u_0\|_{\dot B^{-3/2}_{2,\infty}},$ $E_{0,0}$  and $E_{1,0},$ 
$$\frac d{dt} E_1+2c_0X_1^{7/5}+Y_1\leq CU E_1 +CR^3X_1^3.$$  
Defining $\wt E_1,$ $\wt X_1$ and $\wt Y_1$ as in \eqref{eq:wtE_1}, we discover that whenever $CR^3 X_1^{8/5} \leq c_0,$ we have
$$\frac d{dt} \wt E_1+2c_0\wt X_1^{7/5}+\wt Y_1\leq 0,$$
which leads for all $\beta\in(1,5/2)$ to   \eqref{eq:DDD1}. 
\end{proof}

	
	\subsection{Getting the Lipschitz bound} Our aim here is to establish that \eqref{eq:Lip} is satisfied 
	uniformly with respect to $T>0$
	if  $c_0$ in \eqref{eq:smalldata} is small enough. 
		{To do this, it is convenient to take advantage of the  embedding:
	    \begin{equation}\label{eq:GN} \|z\|_{L^\infty}\lesssim \|\nabla  z\|_{L^{3,1}},   \end{equation}
where the Lorentz space $L^{3,1}$ is defined by real interpolation from Lebesgue spaces as follows (see \eqref{eq:lorentz}): 
\begin{equation}\label{eq:L31}
L^{3,1}:=[L^2,L^\infty]_{1/3,1}.\end{equation}
Recall that for the Stokes system 
$$-\Delta z+\nabla Q=g\andf \div z=0\quad\hbox{in } \  \R^3,$$
we have the following well-known property  of regularity for all $p\in(1,\infty)$:
$$\|\nabla^2 z\|_{L^p}\leq C_p\|g\|_{L^p}.$$
Consequently, leveraging  Proposition \ref{p:interpo},  we deduce that
$$\|\nabla^2 z\|_{L^{3,1}}\leq C_p\|g\|_{L^{3,1}}.$$
Applying this to  \eqref{eq:stokes} yields 
   $$\|\nabla^2u\|_{L^{3,1}}\lesssim \|u_t\|_{L^{3,1}} + \|u\cdot\nabla u\|_{L^{3,1}} + \Bigl\|\int  {(v-u)f}\,dv\Bigr\|_{L^{3,1}}.$$}
   Therefore, using \eqref{eq:GN} guarantees that  for all $T>0,$
      \begin{equation}\label{eq:Lip0}
   \int_0^T\|\nabla u\|_{L^\infty}\,dt \lesssim \int_0^T \biggl(\|u_t\|_{L^{3,1}} + \|u\cdot\nabla u\|_{L^{3,1}} + \Bigl\|\int  {(v-u)f}\,dv\Bigr\|_{L^{3,1}}\biggr)dt.\end{equation} 
   On the one hand,  using \eqref{eq:reiteration} then  
    $\dot H^1\hookrightarrow L^6,$ we discover that
\begin{equation}\label{eq:L4}
	 \|u_t\|_{L^{3,1}} \lesssim \|u_t\|_{L^2}^{1/2}\|u_t\|_{L^6}^{1/2}\lesssim
	 \|u_t\|_{L^2}^{1/2}\|\nabla u_t\|^{1/2}_{L^2}.\end{equation}
On the other hand, from  H\"older inequality, \eqref{eq:GN} and the same argument as for $u_t,$  one can write that
 $$
 \|u\cdot\nabla u\|_{L^{3,1}} \leq \|u\|_{L^\infty}\|\nabla u\|_{L^{3,1}}\lesssim \|\nabla u\|_{L^{3,1}}^2\lesssim \|\nabla u\|_{L^2}\|\nabla^2u\|_{L^2}.
 $$
Finally, using again \eqref{eq:reiteration}, then 
 \eqref{eq:L2}  yields
$$\begin{aligned} \Bigl\|\int  {(v-u)f}\,dv\Bigr\|_{L^{3,1}}&\lesssim   \Bigl\|\int {(v-u)f}\,dv\Bigr\|_{L^2}^{2/3}\Bigl\|\int  {(v-u)f}\,dv\Bigr\|_{L^\infty}^{1/3}\\
 &\lesssim  \|\rho\|_{L^\infty}^{1/3}\biggl(\int\!\!\!\int |v-u|^2f\,dv\,dx\biggr)^{1/3} \bigl(\|j\|_{L^\infty} 
 + \|\rho\|_{L^\infty}\|u\|_{L^\infty}\bigr)^{1/3}\cdotp\end{aligned}$$
 Hence, since
    $$\|u\|_{L^\infty}\lesssim \|\nabla u\|_{L^2}^{1/2}\|\nabla^2 u\|_{L^2}^{1/2},  $$
 we obtain 
 $$\Bigl\|\int  {(v-u)f}\,dv\Bigr\|_{L^{3,1}}
  \leq \|\rho\|_{L^\infty}^{1/3}\biggl(\int\!\!\!\int |v-u|^2f\,dv\,dx\biggr)^{1/3} \bigl(\|j\|_{L^\infty}^{1/3} \!+\! \|\rho\|_{L^\infty}^{1/3}
  \|\nabla u\|_{L^2}^{1/6}  \|\nabla^2 u\|_{L^2}^{1/6}\bigr)\cdotp$$
  Whenever the small Lipschitz norm condition \eqref{eq:Lip} is satisfied and $E_{0,0}$ is small enough, Inequalities \eqref{eq:boundrho}
  and \eqref{eq:bounde} give us
  $$\|\rho\|_{L^\infty}+\|j\|_{L^\infty}\leq C  \|\langle v\rangle ^2 f_0\|_{L^1(\R^3_v;L^\infty(\R^3_x))}.$$
  To complete the proof, let us define $T$ to be the largest time for which \eqref{eq:Lip} holds true. 
  Reverting to \eqref{eq:Lip0} and remembering the definition of $D_0,$ $D_1$ and $D_2,$  we discover that 
   we have    for some constant $C_R$ depending only on  $\|\langle v\rangle ^2 f_0\|_{L^1(\R^3_v;L^\infty(\R^3_x))},$ 
  $$  \int_0^T\|\nabla u\|_{L^\infty}\,dt \leq C_R\int_0^T \bigl(D_1^{1/4}D_2^{1/4} 
  + D_0^{1/2}D_1^{1/2} +  
   D_1^{1/3}  + D_0^{1/12} D_1^{5/12}\bigr)dt.$$  
   Let us take $\alpha$ close enough to $3/2$   in \eqref{eq:DDD0} and \eqref{eq:DDD1},  and observe  that the above inequality 
  may be rewritten
   \begin{multline*}  \int_0^T\|\nabla u\|_{L^\infty}\,dt \leq C_R
   \int_0^T \bigl(t^{-1/4}(1\!+\!a_0t)^{-(\alpha+1)/4}   ((1\!+\!a_0t)^{\alpha+1}D_1)^{1/4}(tD_2)^{1/4} + D_0^{1/2}D_1^{1/2}\\
   +(1\!+\!a_0t)^{-(\alpha+1)/3} ((1\!+\!a_0t)^{\alpha+1} D_1)^{1/3}
  +(1\!+\!a_0t)^{-\frac{6\alpha+5}{12}}((1\!+\!a_0t)^\alpha D_0)^{1/12} 
  ((1\!+\!a_0t)^{\alpha+1} D_1)^{5/12}\bigr)dt.  \end{multline*}
 Using H\"older inequality, Inequality \eqref{eq:cE} and the decay estimates \eqref{eq:DDD0} and \eqref{eq:DDD1}, 
 one can conclude that the right-hand side may be bounded in terms of the data 
 if we take\footnote{In accordance with Remark \ref{r:1}, if assuming that $u_0\in\dot B^{-\sigma}_{2,\infty},$
  the same approach would give  $\alpha=\sigma,$ whence  the restriction $\sigma>1.$}   $\alpha\in(1,3/2).$  This gives 
 $$\int_0^T\|\nabla u\|_{L^\infty}\leq C_0$$ with $C_0$ depending only 
 on $R$ and on the norms of the initial data and tending to $0$ when $E_{0,0}+E_{1,0}$ tends to $0.$ 


\section{Proving theorem \ref{thm:main1}} \label{s:thm1}

This section is devoted to the proof of our first main result.
We shall first present a stability result that will entail the uniqueness part of   Theorem \ref{thm:main1}, and will be also used in 
the subsection dedicated to the existence.  The last subsection is devoted to the study 
of the large time behavior of the distribution function.

	\subsection{A stability result}\label{ss:stab} 

Since the result we have in mind will be useful for both the proof of uniqueness and strong convergence of the solutions 
to some approximate (VNS) system (see the next part),  it is convenient to consider 
 the following linear system:
 \begin{equation}\label{eq:VNS}\left\{\begin{aligned}&\d_tf+v\cdot\nabla_xf+\divv\bigl({(w-v)f}\bigr)=0,\\
&\d_tu-\Delta u+ w\cdot\nabla u+\int {(w-v)f}\,dv+\nabla P= S,\\
&\div v=0,\end{aligned}\right.\end{equation}
where the transport field $w$ and the source term $S$ are given sufficiently smooth time-dependent vector-fields. 
\smallbreak
Our goal is to prove a stability estimate for \eqref{eq:VNS}, that implies
a uniqueness result for (VNS) in the functional framework of Theorem \ref{thm:main1}. To do this, 
we  adapt the approach  of \cite{HMMM} . Let us underline that the proof here is simpler since
the characteristics associated to the first equation of \eqref{eq:VNS} are Lipschitz. 

So, let us consider two solutions $u_1$ and $u_2$ of \eqref{eq:VNS} associated with vector-fields 
$w_1$ and $w_2,$ source terms $S_1$ and $S_2,$ and  initial velocities $u_{0,1}$ and $u_{0,2}$.
For simplicity,  we take the same  initial distribution function $f_0$  for the two solutions.  
 
  Let us set $\df:=f_2-f_1,$ $\du:=u_2-u_1,$ $\dw:=w_2-w_1$   and $\dS:=S_2-S_1.$  
We denote by $\rho_i$ and $m_{2,i}$ the density (resp. energy) associated with $f_i$ through \eqref{eq:not}.
We also introduce the characteristic curves $Z_1=(X_1,V_1)$ and $Z_2=(X_2,V_2)$ associated with vector-fields  $w_1$ and $w_2,$ respectively,
in System \eqref{eq:car}. We use the notation $Y_j(t,x,v):= Y_j(t;0,x,v)$ for $Y\in\{V,X,Z\}.$ Finally, we  set:
$$\dX:=X_2-X_1,\quad \dV:=V_2-V_1\andf \dZ:=Z_2-Z_1.$$ 
Still denoting by $\cP$ the Leray projector, the system satisfied by $\du$ reads:
\begin{equation}\label{eq:du}\left\{\begin{array}{l}
\d_t\du+w_1\cdot\nabla\du-\Delta\du+\rho_2\dw =\cP\biggl(\dS-\dw\cdot\nabla u_2+\Int {(v-w_1)\df}\,dv\biggr),\\[1ex]
\div\du =0.\end{array}\right.\end{equation}
Taking the $L^2(\R^3;\R^3)$ scalar product with $\du$ immediately gives
\begin{multline}\label{eq:du0}
\frac12\frac d{dt}\|\du\|_{L^2}^2+\|\nabla\du\|_{L^2}^2 +\|\sqrt{\rho_2}\,\du\|_{L^2}^2\\=-\int(\dw\cdot\nabla u_2)\cdot \du\,dx
+\int\!\!\!\int {((v-w_1)\df)}\cdot\du\,dv\,dx+\int \du\cdot\dS\,dx.\end{multline}
For the first term of the right-hand side, combining H\"older, Sobolev  and Young inequalities yields
\begin{equation}\label{eq:uniq1}
-\int(\dw\cdot\nabla u_2)\cdot \du\,dx\leq\frac1{10}\|\nabla\du\|_{L^2}^2+C\|\nabla u_2\|_{L^3}^2\|\dw\|_{L^2}^2.
\end{equation}
To handle the second term (let us call it $\dF$), using Formula \eqref{eq:formularho} for $f_1$ and $f_2$ gives
\begin{equation}\label{eq:dF}
\dF(t)
=e^{3t}\int\!\!\!\int\bigl(f_0(Z_2^{-1}(t))-f_0(Z_1^{-1}(t))\bigr)(v-w_1)\cdot \du\,dv\,dx.\end{equation}
Performing the changes of variables $(x',v')=Z_2^{-1}(t,x,v)$ and $(x',v')=Z_1^{-1}(t,x,v),$ respectively
and using the properties of the flow (see Appendix \ref{s:AppendixA}), we obtain 
$$\displaylines{\dF=\int\!\!\int \bigl(V_2(t)-w_1(t,X_2(t))\bigr)\cdot\du(t,X_2(t))\,f_0\,dv'\,dx'\hfill\cr\hfill
-\int\!\!\int \bigl(V_1(t)-w_1(t,X_1(t))\bigr)\cdot\du(t,X_1(t))\,f_0\,dv'\,dx'.}$$
We further decompose $\dF$  into $A_1+A_2+A_3+A_4$ with 
$$\begin{aligned}
A_1(t)&:=\int\!\!\!\int f_0(x,v)\dV(t,x,v)\cdot\du(t,X_2(t,x,v))\,dv\,dx\\
A_2(t)&:=\int\!\!\!\int f_0(x,v)\bigl(w_1(t,X_1(t,x,v))-w_1(t,X_2(t,x,v))\bigr)\cdot\du(t,X_1(t,x,v))\,dv\,dx\\
A_3(t)&:=\int\!\!\!\int f_0(x,v)w_1(t,X_2(t,x,v))\bigl(\du(t,X_1(t,x,v))-\du(t,X_2(t,x,v))\bigr)\,dv\,dx\\
A_4(t)&:=\int\!\!\!\int f_0(x,v)V_1(t,x,v)\cdot\bigl(\du(t,X_2(t,x,v))- \du(t,X_1(t,x,v))\bigr) \,dv\,dx.
\end{aligned}$$
For bounding $A_1,$ we  use Young inequality:
$$A_1\leq\frac12\|\sqrt{f_0}\,\dV\|_{L^2}^2+\frac12\int\!\!\! f_0(x,v)|\du(t,X_2(t,x,v))|^2\,dx\,dv.$$
From \eqref{eq:formularho}, we  deduce that 
$$
\int\!\!\! f_0(x,v)|\du(t,X_2(t,x,v))|^2\,dx\,dv=\int\!\!\!\int f_2(t,x,v)|\du(t,x)|^2\,dv\,dx=\|\sqrt{\rho_2}\du(t)\|_{L^2}^2.
$$
Hence 
\begin{equation}\label{eq:A1}
A_1\leq \frac12\|\sqrt{f_0}\,\dV\|_{L^2}^2+\frac12\|\sqrt{\rho_2}\du(t)\|_{L^2}^2.\end{equation}
Next, using  Cauchy-Schwarz inequality and a change of variable gives
$$\begin{aligned}
A_2&\leq\|\nabla w_1\|_{L^\infty}\int\!\!\!\int f_0(x,v)|\dX(t,x,v)||\du(t,X_1(t,x,v))|\,dvdx\\
&\leq \|\nabla w_1\|_{L^\infty} \|\sqrt f_0 \dX\|_{L^2}\biggl(\int\!\!\!\int f_0(x,v) |\du(t,X_1(t,x,v))|^2\,dvdx\biggr)^{1/2}\\
&\leq \|\nabla w_1\|_{L^\infty} \|\sqrt f_0 \dX\|_{L^2}\biggl(\int\!\!\!\int f_1(t,x,v) |\du(t,x)|^2\,dvdx\biggr)^{1/2}.
\end{aligned}$$
Hence we have 
\begin{equation}\label{eq:A2}
A_2\leq\|\nabla w_1\|_{L^\infty}  \|\sqrt f_0 \dX\|_{L^2}\|\sqrt{\rho_1}\,\du\|_{L^2}.
\end{equation}
In order to bound $A_3$ and $A_4,$ we need to resort to the following inequality 
(see \cite{Stein}) that involves Hardy's maximal function $\cM:$
\begin{equation}\label{eq:Maximal}
|z(y)-z(x)|\leq C|x-y|\bigl(\cM(\nabla z)(y)+\cM(\nabla z)(x)\bigr)\cdotp\end{equation}
This enables us to write that\begin{multline}\label{eq:uniq2}
A_3\lesssim\int\!\!\! f_0(x,v) |w_1(t,X_2(t,x,v))| \bigl(\cM(\nabla \du)(t,X_1(t,x,v))\\+\cM(\nabla \du)(t,X_2(t,x,v))\bigr)
\dX(t,x,v)\,dxdv.\end{multline}
Now, from the usual change of variable, \eqref{eq:formularho} and the continuity of the maximal function on $L^2,$
 we gather that for $j=1,2,$ 
$$\int\!\!\!\int  (\cM(\nabla \du )(X_j))^2\,f_0\,dvdx=  \int (\cM(\nabla\du ))^2\biggl(\int f_j\,dv\biggr)dx\lesssim \|\rho_j\|_{L^\infty}\|\nabla\du\|_{L^2}^2.
$$
Hence, reverting to \eqref{eq:uniq2} gives
\begin{equation}\label{eq:A3}
A_3\leq \frac1{10}\|\nabla\du\|_{L^2}^2 + C\|w_1\|_{L^\infty}^2\|(\rho_1,\rho_2)\|_{L^\infty}  \|\sqrt f_0 \dX\|_{L^2}^2.
\end{equation}
Term $A_4$ is the most involved. To handle it, we observe that $A_4\leq A_{4,1}+A_{4,2}+A_{4,3}$ 
where 
$$
A_{4,3}= \int\!\!\!\int f_0(x,v)|\dV(t,x,v)|\,|\dX(t,x,v)| \cM(\nabla \du)(t,X_2(t,x,v))\,dvdx$$
and, for $j=1,2,$
$$
A_{4,j}=\int\!\!\!\int f_0(x,v)  |\dX(t,x,v)| |V_j(t,x,v)| \cM(\nabla\du)(t,X_j(t,x,v))\,dvdx.$$
By Cauchy-Schwarz inequality, change of variable, formula  \eqref{eq:formularho} and the continuity 
in $L^2$ of the maximal function, we have for $j=1,2,$ 
\begin{align}\label{eq:A4j}
A_{4,j}&\leq \|\sqrt{f_0}\,\dX\|_{L^2}\biggl(\int\!\!\!\int f_0|V_j|^2 (\cM(\nabla\du) (X_j))^2\,dvdx\biggr)^{1/2}\nonumber\\
&\leq \|\sqrt{f_0}\,\dX\|_{L^2}\|m_{2,j}\|_{L^\infty}^{1/2} \|\cM(\nabla\du)\|_{L^2}\nonumber\\
&\leq\frac1{10}\|\nabla\du\|_{L^2}^2+C\|m_{2,j}\|_{L^\infty}\|\sqrt{f_0}\, \dX\|_{L^2}^2.\end{align}
Similarly, 
\begin{equation}\label{eq:A43}
A_{4,3}\leq \frac1{10}\|\nabla\du\|_{L^2}^2+C\|\rho_2\|_{L^\infty}\|\dV\|_{L^\infty}^2\|\sqrt{f_0}\, \dX\|_{L^2}^2.
\end{equation}
Putting  together Inequalities \eqref{eq:uniq1}, \eqref{eq:A1}, \eqref{eq:A2}, 
  \eqref{eq:A3}, \eqref{eq:A4j} and \eqref{eq:A43}  yields  
\begin{multline}\label{eq:uniq5}
\frac d{dt}\|\du\|_{L^2}^2+\|\nabla\du\|_{L^2}^2 +\|\sqrt{\rho_2}\,\du\|_{L^2}^2 \leq 
2\int \dS\cdot \du\,dx
+\|\sqrt{f_0}\,\dV\|_{L^2}^2\\
+C\Bigl(\|\nabla u_2\|_{L^3}^2\|\dw\|_{L^2}^2+ \|\rho_1\|_{L^\infty}^{1/2}  \|\nabla w_1\|_{L^\infty}
\|\sqrt{f_0}\,\dX\|_{L^2}\|\du\|_{L^2}\\+
\bigl(\|(\rho_1,\rho_2)\|_{L^\infty}\|w_1\|_{L^\infty}^2 +\|(m_{2,1},m_{2,2})\|_{L^\infty}
+\|\rho_2\|_{L^\infty}\|\dV\|_{L^\infty}^2\bigr) \|\sqrt{f_0}\,\dX\|_{L^2}^2\Bigr)\cdotp
\end{multline}
We have to keep in mind that \eqref{eq:V}  implies that 
\begin{equation}\label{eq:uniq3} \dV(t,x,v)=\int_0^t e^{s-t}\bigl(w_2(s, X_2(s,x,v))-w_1(s,X_1(s,x,v))\bigr)\,ds,
\end{equation}
and thus 
\begin{equation}\label{eq:uniq4}
\|\dV\|_{L^2(0,t;L^\infty)} \leq \|w_1\|_{L^2(0,t;L^\infty)} +\|w_2\|_{L^2(0,t;L^\infty)} .\end{equation}
To complete the proof, we have to look at the time evolution of $\|\sqrt{f_0}\,\dZ(t)\|_{L^2(\R^3\times\R^3)}.$
Using \eqref{eq:car}, we see that $$\dX_t=\dV \andf\dV_t=w_2(X_2)-w_1(X_1)-\dV.$$ Hence
$$\begin{aligned}\frac12\frac d{dt}\|\sqrt{f_0}\,\dZ\|_{L^2}^2\!+\!\|\sqrt{f_0}\,\dV\|_{L^2}^2&=\!\int\!\!\!\int\! f_0\dX\!\cdot\!\dV\,dvdx
+\!\int\!\!\!\int\! f_0 \bigl(w_2(X_2)-w_1(X_1)\bigr)\!\cdot\!\dV dvdx\\
&\leq \frac12\|\sqrt{f_0}\,\dX\|_{L^2}^2+\frac12\|\sqrt{f_0}\,\dV\|_{L^2}^2+B_1+B_2\end{aligned}$$
with 
$$B_1:=\int\!\!\!\int\! f_0 \bigl(w_1(X_2)-w_1(X_1)\bigr)\!\cdot\!\dV dvdx
\andf B_2:=\int\!\!\!\int\! f_0 \bigl(w_2(X_2)-w_1(X_2)\bigr)\!\cdot\!\dV dvdx.$$
It is obvious that 
\begin{align}\label{eq:f01}
B_1&\leq \|\nabla w_1\|_{L^\infty} \int\!\!\!\int\! f_0|\dX||\dV|\,dv\,dx,\\
B_2&\leq \frac12\int\!\!\!\int f_0|\dV|^2
+\frac12\int\!\!\!\int f_0(x,v) |\dw(t,X_2(t))|^2\,dvdx\nonumber\\ \label{eq:f02}
&\leq \frac12\int\!\!\!\int f_0|\dV|^2
+\frac12\int\!\!\!\int f_2(t,x,v) |\dw(t)|^2\,dvdx.\end{align}
Hence, we end up with 
\begin{equation}\label{eq:uniq13}
\frac d{dt}\|\sqrt{f_0}\,\dZ\|_{L^2}^2
\leq(1+\|\nabla w_1\|_{L^\infty})\|\sqrt{f_0}\,\dZ\|_{L^2}^2+\|\sqrt{\rho_2}\,\dw\|_{L^2}^2.
\end{equation}
Putting this inequality together with \eqref{eq:uniq4} and \eqref{eq:uniq5} yields
\begin{multline}\label{eq:stab} \frac d{dt}\bigl(\|\du\|_{L^2}^2+\|\sqrt{f_0}\,\dZ\|_{L^2}^2\bigr)+
\|\nabla\du\|_{L^2}^2+\|\sqrt{\rho_2}\,\du\|_{L^2}^2
\leq 2\int \dS\cdot \du\,dx+\|\sqrt{\rho_2}\,\dw\|_{L^2}^2\\+
C\biggl(\|\nabla u_2\|_{L^3}^2\|\du\|_{L^2}^2+ \bigl((1+\|\rho_1\|_{L^\infty}^{1/2})\|\nabla w_1\|_{L^\infty}+1\bigr) \bigl(\|\du\|_{L^2}^2+\|\sqrt{f_0}\,\dZ\|_{L^2}^2\bigr)\\+
\bigl(\|(\rho_1,\rho_2)\|_{L^\infty}\|w_1\|_{L^\infty}^2 +\|(m_{2,1},m_{2,2})\|_{L^\infty}
+\|\rho_2\|_{L^\infty}\|(w_1,w_2)\|_{L^2(0,t;L^\infty)}\bigr) \|\sqrt{f_0}\,\dX\|_{L^2}^2\biggr)\cdotp
\end{multline}
By Gronwall lemma the above inequality readily implies the uniqueness part of Theorem \ref{thm:main1} 
(take $S_1=S_2=0,$ $w_1=u_1$ and $w_2=u_2$). Indeed, we then have
$\nabla u_1$  in $L^1_{loc}(\R_+;L^\infty),$ 
$\nabla u_2$  in $L^2(\R_+;L^3),$ $u_1,u_2$ belong to  $L^2_{loc}(\R_+;L^\infty)$ and 
$\rho_1,\rho_2,m_{2,1},m_{2,2}$  are in $L^\infty_{loc}(\R_+;L^\infty).$ 


   \subsection{The proof of existence}\label{s:construction}

Fix some  initial data $(f_0,u_0)$ satisfying  the smallness hypothesis \eqref{eq:smalldata} and
 consider the following Friedrichs approximation of the Vlasov-Navier-Stokes system:   
$$\left\{\begin{array}{l}
\d_tf+v\cdot\nabla_xf+\divv\bigl({(J_nu-v)f}\bigr)=0,\\[1ex]
\d_t u+J_n(u\cdot\nabla_x u)-\Delta_x u=J_n\biggl(\Int {(v-u)f}\,dv\biggr),\end{array}\right.\leqno(VNS_n)$$
where   the spectral orthogonal projector $J_n$ of $L^2(\R^3;\R^3)$ is defined   for all $n\in\N$ by  $J_n=\cP 1_{B(0,n)}(D).$  
Note that the range of $J_n$ is the space 
$$L^2_n:=\bigl\{z\in L^2(\R^3;\R^3),\;  \div z=0\andf {\rm Supp}\, \wh z\subset \bar B(0,n)\bigr\}\cdotp$$
Let us admit temporarily that System $(VNS_n)$  supplemented with initial data $(f_0,J_nu_0)$
has a unique solution $(f^n,u^n)$ on some time interval $[0,T]$ such that $f^n$ satisfies the properties stated in 
Theorem \ref{thm:main1} and $u^n$ belongs to   the space 
$$E_T^n:=\cC([0,T]; L^2_n).$$ 
Owing to the spectral localization, the vector-field $u^n$ belongs to all Sobolev spaces and 
since  $J_nu^n=u^n$ and   $J_n$ is a $L^2$ orthogonal projector, all 
 the energy estimates of  Sections \ref{s:energy}  are valid for $(f^n,u^n).$ 
{In particular, there holds that
 \begin{multline}\label{eq:energyn}
 \frac12\Bigl(\|u^n(t)\|_{L^2(\R^3_x)}^2+\||v|^2f^n\|_{L^1(\R^3_x\times\R^3_v)}\Bigr)\\
 +\int_0^t\Bigl(\|u^n(t)\|_{L^2(\R^3_x)}^2+\||v-u^n|^2f^n\|_{L^1(\R^3_x\times\R^3_v)}\Bigr)\leq E_{0,0}\end{multline}}
 and,    from Corollary \ref{c:final}, we infer  that  if \eqref{eq:Lip} holds on $[0,T],$ then
we have
\begin{equation}\label{eq:cEn}
\cE^n(t)+\int_0^t\cD^n\,d\tau\leq C\biggl(\bigl(1+\|\langle v\rangle^2 f_0\|_{L^1(\R^3_v;L^\infty(\R^3_x))}\bigr)\biggl(\|u_0\|_{H^1}^2 + \int\!\!\!\int |v|^2f_0\,dv\,dx\biggr)\biggr)
\end{equation} 
for all $t\in[0,T],$ where the functionals $\cD^n$ and $\cE^n$ corresponding to $(f^n,u^n)$ are defined as in \eqref{eq:DE}.  
\smallbreak
{From the definition of Besov norms in \eqref{def:besov} and Fourier-Plancherel theorem, 
it is clear that  $J_n$ maps all  Besov spaces 
$\dot B^{-\sigma}_{2,\infty}$ to themselves, with norm $1.$  One can thus repeat the computations of  Section \ref{s:decay} and 
eventually get the key Lipschitz control \eqref{eq:Lip}.}
  This allows us to continue $(f^n,u^n)$  for all positive times {(see \eqref{eq:continuation} below),} and \eqref{eq:cEn} is satisfied on $\R_+.$
\smallbreak
In short, we have constructed a sequence $(f^n,u^n)_{n\in\N}$ of global approximate solutions, that 
satisfy \eqref{eq:cEn} and all the estimates of Sections \ref{s:energy} and \ref{s:decay}, uniformly 
with respect to $n$ (since the right-hand side may be bounded independently of $n$).  
 Furthermore, the distribution $f^n$ satisfies the estimates \eqref{eq:formularho} to  \eqref{eq:boundrho}, in particular
 \begin{equation}\label{eq:boundfn} 
 \|f^n(t)\|_{L^\infty(\R^3\times\R^3)} \leq e^{3t} \|f_0\|_{L^\infty(\R^3\times\R^3)}\quad\hbox{for all }\ t\geq0.\end{equation}


From Inequalities \eqref{eq:cEn} and \eqref{eq:boundfn}, we deduce that, up to subsequence, 
$$
f^n\rightharpoonup f\quad\hbox{in }\ L^\infty \andf
u^n\rightharpoonup u \quad\hbox{in }\ L^\infty(\R_+;H^1)\cap L^2_{loc}(\R_+;H^2) \qquad  \hbox{weak } *.$$
We claim  that $(u^n)_{n\in\N}$ converges \emph{strongly} to $u$ in the energy space. 
In fact, we are going to show that it is a Cauchy sequence in $\cC([0,T];L^2) \cap L^2(0,T; H^1)$ for all 
$T>0.$ To prove it, we observe that for all $n\in\N,$ the velocity field $u^n$ satisfies 
the second equation of $(VNS)$ with source term 
$$S^n:=(\cP-J_n)(u^n\cdot\nabla u^n)+(\cP-J_n)\int {(u^n-v)f^n}\,dv.$$ 
Now, owing to the spectral cut-off, Inequality \eqref{eq:L2}, H\"older  and Gagliardo-Nirenberg inequality, we have
$$\begin{aligned}
\|S^n\|_{\dot H^{-1}} &\leq n^{-1}\biggl(\|u^n\cdot\nabla u^n\|_{L^2}+\biggl\|\int  {(u^n-v)f^n}\,dv\biggr\|_{L^2}\biggr)\\
&\leq Cn^{-1}\biggl(\|\nabla u^n\|_{L^{2}}^{3/2}\|\nabla^2u^n\|_{L^2}^{1/2} +\|\rho^n\|_{L^\infty}^{1/2}
\biggl(\int\!\!\!\int |u^n-v|^2f^n\,dv\,dx\biggr)^{1/2}\biggr)\cdotp
\end{aligned}$$
Hence, introducing the functionals $E_0^n,$ $D_0^n,$ $E_1^n$ and $D_1^n$ like in  \eqref{eq:E} and \eqref{eq:D},
 we discover that
for all $t\geq0,$
$$\int_0^t\|S^n\|_{\dot H^{-1}}^2\,d\tau\leq Cn^{-2}\biggl(\int_0^t \bigl(D_0^n D_1^n)^{1/2} E_1^n + \|f_0\|_{L^1(\R^3_v;L^\infty(\R^3_x))}
D_0^n\bigr)d\tau\biggr)\cdotp$$
Note also that 
$$\|J_nu_0-J_mu_0\|_{L^2} \leq n^{-1}\|u_0\|_{\dot H^1}  \quad\hbox{for } m>n>0.$$
Therefore, taking advantage of Inequalities \eqref{eq:stab} and \eqref{eq:cEn} and of Gronwall lemma, 
one can  conclude that 
for all $n\in\N$ and $m>n,$ we have for some increasing function $C_{0}$ depending on the data, but 
\emph{independent of $m$ and $n$}, 
$$\|(u^m-u^n)(t)\|_{L^2}^2+\int_0^t\|\nabla(u^m-u^n)\|_{L^2}^2\,d\tau \leq C_{0}(t)\: n^{-2}.$$
This completes the proof of our claim.
\medbreak 
Since the convergence of $(u^n)_{n\in\N}$ is strong in the energy space, it is now easy to conclude that $(f,u)$ satisfies
(VNS) in the sense of distributions.  
Furthermore,  since for all $n\in\N$, the approximate solution $(f^n,u^n)$ satisfies  the energy balance
$$\displaylines{\frac12\biggl(\|u^n(t)\|_{L^2}^2+\int\!\!\! \int  |v|^2f^n(t)\,dv\,dx\biggr) +\int_0^t\biggl(\|\nabla u^n\|_{L^2}^2+\int\!\!\!\int f|v-u^n|^2\,dv\,dx\biggr)d\tau\hfill\cr\hfill 
= \frac12\biggl(\|u^n_0\|_{L^2}^2+\int\!\!\! f_0 |v|^2\,dv\,dx\biggr),}$$ 
the strong convergence guarantees that $(f,u)$ satisfies  \eqref{eq:energy0}. 
\smallbreak
Finally,  that $u\in \cC(\R_+;H^1)$ stems from the fact that 
$u$ satisfies a heat equation with initial data in $H^1$ and right-hand side in $L^2_{loc}(\R_+;L^2).$  
\medbreak 
For completeness, let us explain how to solve $(VNS_n)$ locally in time for fixed $n\in\N.$  
Toward this, we consider the map $\Phi$ defined on $E_T^n$ by $\Phi(w)=u,$ where $u$ stands  for the solution of the \emph{linear}
parabolic equation
\begin{equation}\label{eq:Sn}
\d_tu-\Delta u=J_n\biggl(\int {(v-w)f}\,dv -w\cdot\nabla u\biggr),\qquad u|_{t=0}=J_nu_0,
\end{equation}
and  $f$ stands for the unique bounded solution of the (linear) transport equation~:
\begin{equation}\label{eq:f}
\d_tf+v\cdot\nabla_xf+\divv\bigl({(w-v)f}\bigr)=0,\qquad f|_{t=0}=f_0.
\end{equation}
To implement the fixed point theorem, let us endow $E^n_T$ with the following norm~:
$$\|z\|_{E^n_T}:=\biggl(\underset{t\in[0,T]}\sup \|z(t)\|_{L^2}^2 +\int_0^T\|\nabla z\|_{L^2}^2\,dt\biggr)^{1/2}.$$
Since $w\in E^n_T,$ we actually have $w\in\cC([0,T];C^{1,0})$ owing to the spectral localization. This
guarantees the existence and uniqueness of a solution to \eqref{eq:f}: it is given by \eqref{eq:formularho}
after computing the  characteristics according to \eqref{eq:car} (with $w$ instead of $u,$ of course). 
\smallbreak
We claim that if $T$ is small enough, then $\Phi$ has a fixed point in $E_T^n.$ 
Indeed, by an energy method, we get
$$\frac12\frac d{dt}\biggl(\|u\|_{L^2}^2+\int\!\!\!\int |v|^2f\,dv\,dx\biggr)+\|\nabla u\|_{L^2}^2+\int\!\!\!\int |v-u|^2f\,dv\,dx
=\int\!\!\!\int  {((w-u)f)}\cdot (v- u)\,dv\,dx.$$
Combining H\"older, Sobolev and Young inequalities gives
$$\int\!\!\!\int {((w-u)f)}\cdot (v-u)\,dv\,dx\leq \int\!\!\!\int |v-u|^2f\,dv\,dx+\frac14\|\rho\|_{L^\infty}\|w-u\|_{L^2}^2
$$
where $\rho$ denotes the density associated to $f$ through \eqref{eq:not}. 
\smallbreak
To continue the proof, we assume that $w$ satisfies \eqref{eq:Lip}, so that 
we have  $$\|\rho\|_{L^\infty(0,T;L^{\infty})}\leq 2\|f_0\|_{L^1(\R^3_v;L^{\infty}(\R^3_x))}.$$
Then, the above inequality combined with Gronwall's lemma gives for all $t\in[0,T],$
\begin{multline}\label{eq:pf1}
\|u(t)\|_{L^2}^2+2\int_0^t\|\nabla u\|_{L^2}^2\,d\tau\\\leq \biggl(\|u_0\|_{L^2}^2+\int\!\!\!\int |v|^2f_0\,dv\,dx
+t\|f_0\|_{L^1(\R^3_v;L^{\infty}(\R^3_x))}\|w\|_{L^\infty(0,t:L^2)}^2\biggr)
e^{t\|f_0\|_{L^1(\R^3_v;L^{\infty}(\R^3_x))}}\cdotp\end{multline}
Let us assume that $w$ belongs to the closed ball $\bar B_{E_T^n}(0,M)$ of $E_T^n$ with $M$ defined by
\begin{equation}\label{def:M} M^2=2\biggl(1+\|u_0\|_{L^2}^2+\int\!\!\!\int |v|^2f_0\,dv\,dx\biggr)\cdotp\end{equation}
The spectral localization of $w$  and a suitable Gagliardo-Nirenberg inequality guarantee that
$$\|\nabla w\|_{L^\infty}\leq C\|\nabla w\|_{\dot H^1}^{1/2} \|\nabla w\|_{\dot H^2}^{1/2}\leq Cn^{3/2}\|\nabla w\|_{L^2},
$$
whence
$$\int_0^T \|\nabla w\|_{L^\infty} \leq Cn^{3/2} T^{1/2}\|w\|_{E_T^n}.$$
Since we need  \eqref{eq:Lip} to be satisfied by $w,$ we conclude that if $T$ is such that
\begin{equation}\label{eq:smallT}Cn^{3/2} T^{1/2} M\leq \delta,\quad T \|f_0\|_{L^1(\R^3_v;L^{\infty}(\R^3_x))}\leq\log2\andf \|f_0\|_{L^1(\R^3_v;L^{\infty}(\R^3_x))}TM^2\leq 1,\end{equation}
then $u$ is in $\bar B_{E_T^n}(0,M),$ too, owing to \eqref{eq:pf1}. 
\smallbreak
In order to complete the proof of the existence of a fixed point in $\bar B_{E_T^n}(0,M),$  it suffices
to show that $\Phi$ is a contracting mapping on this ball. 
To do so, we consider two elements $w_1$ and $w_2$ of $\bar B_{E_T^n}(0,M)$ with $T$ satisfying \eqref{eq:smallT}. 
Then, both  $u_1:=\Phi(w_1)$ and $u_2:=\Phi(w_2)$  are in $\bar B_{E_T^n}(0,M).$  Let $f_i$ be the distribution associated to $w_i$ through the transport
equation \eqref{eq:f}. 
In order to get the desired property of contraction, one can  adapt the method leading to the stability 
estimate \eqref{eq:stab}, to the following system verified by $(f_i,u_i)$ with $i=1,2$:
$$\left\{\begin{aligned}
&\d_tu_i-\Delta u_i+ J_n\biggl(w_i\cdot\nabla u_i+\int {(w_i-v)f_i}\,dv\biggr)= 0,\\
&\d_tf_i+v\cdot\nabla_xf_i+\divv\bigl({(w_i-v)f_i}\bigr)=0.\end{aligned}\right.$$
The only difference compared to System \eqref{eq:VNS} is   the orthogonal projector $J_n$ in the first equation. 
However, $J_n$ restricted to $L^2_n$ is just the identity so that it has no effect in the energy estimates, and \eqref{eq:stab} still holds. 
Now, remembering that \eqref{eq:Lip} and \eqref{eq:smallT} are satisfied, and taking advantage of the control on $\rho_i$ and
$m_{2,i}$ provided by \eqref{eq:boundrho} and \eqref{eq:bounde}, we get
$$\frac d{dt}\dU+\dH \leq C\bigl(\|\nabla u_2\|_{L^3}^2+\|\rho_2\|_{L^\infty}\bigr)\dW +
C_{f_0}\bigl(1+\|w_1\|_{L^\infty}^2 +\|\nabla w_1\|_{L^\infty}\bigr)\dU,
$$
with $C_{f_0}$ depending only on $\|f_0\|_{L^1(\R^3_v;L^{\infty}(\R^3_x))},$
$$\dU(t):=\|\du(t)\|_{L^2}^2+\|\sqrt{f_0} \dZ(t)\|_{L^2},\quad
\dW(t):=\|\dw(t)\|_{L^2}^2\andf \dH(t):=\|\nabla\du(t)\|_{L^2}^2.$$
Hence, using Gronwall lemma gives
$$
\dU(t)+\int_0^t\dH\,d\tau\leq C\int_0^t e^{C_{f_0}\int_\tau^t(1+\|w_1\|_{L^\infty}^2 +\|\nabla w_1\|_{L^\infty})\,d\tau'}
\bigl(C_{f_0}+\|\nabla u_2\|_{L^3}^2\bigr)\dW\,d\tau.$$
Remember that $w_1$ satisfies \eqref{eq:Lip}. Furthermore, the spectral localization of $w_1$ and $u_2$  ensures that
$$\|w_1\|_{L^\infty}\leq Cn^{1/2}\|w_1\|_{L^2}\andf \|\nabla u_2\|_{L^3}^2\leq Cn^3\|u_2\|_{L^2}^2.$$
Hence, for small enough $t,$ we have
$$\|\du\|_{L^\infty(0,t;L^2)}^2 + \int_0^t \|\nabla\du\|_{L^2}^2\,d\tau\leq 2Ct(C_{f_0}+n^3M)\|\dw\|_{L^\infty(0,t;L^2)}^2,$$
which implies that the map $\Phi$ is contractive in $\bar B_{E_T^n}(0,M),$ if $T$ satisfies \eqref{eq:smallT} and
\begin{equation}\label{eq:smallTbis}
 2CT(C_{f_0}+n^3M)<1.\end{equation}
The Banach fixed point theorem allows to conclude that
$\Phi$ admits a (unique) fixed point in $\bar B_{E_T^n}(0,M).$ 
{Furthermore, owing to \eqref{eq:smallT} and \eqref{eq:smallTbis}, $T$ may be bounded from below
in terms of $E_{0,0},$  $\|f_0\|_{L^1(\R^3_v;L^{\infty}(\R^3_x))}$ and $n.$ Consequently, whenever 
\begin{equation}\label{eq:continuation}\underset{t\in[0,T]}\sup E_0(t)+ \underset{t\in[0,T]}\sup \|f(t)\|_{L^1(\R^3_v;L^{\infty}(\R^3_x))}<\infty,\end{equation}
the solution of $(VNS_n)$ may be continued beyond $T.$}\qed

   
   \subsection{Convergence of the distribution of particles} 
  
  Here we prove Corollary \ref{c:was}. Having  the time decay of $E_1$ at hand, establishing the convergence of the distribution function $f$ 
 may be done as in \cite{HMM}.  As a first, we write that  by definition of $\rho$ and of the Wasserstein distance $W_1,$ we have
  $$W_1(f(t),\rho(t)\otimes\delta_{u(t)})=\underset{\|\nabla_{x,v}\phi\|_{L^\infty}=1}\sup
  \int\!\!\!\int f(t,x,v)\bigl(\phi(x,v)-\phi(x,u(t,x))\bigr)dv\,dx.  $$
  Using the definition of the Lipschitz norm, then Cauchy-Schwarz inequality and \eqref{eq:M0}, one can write
  for all function $\phi$ with Lipschitz norm equal to $1,$ 
  $$\begin{aligned}
   \int\!\!\!\int f(t,x,v)\bigl(\phi(x,v)-\phi(x,u(t,x))\bigr)dx\,dv&\leq \int\!\!\!\int f(t,x,v)|v-u(t,x)|\,dv\,dx\\
   &\leq \biggl(\int\rho(t)\,dx\: \int\!\!\!\int f(t)|v-u(t)|^2\,dv\,dx\biggr)^{1/2}\\
   &\leq \sqrt{ M_0\: E_1(t)}.
  \end{aligned}  $$
  Keeping \eqref{eq:E1t} in mind completes the proof of the decay for the Wasserstein distance. 
  \smallbreak
Similarly, by definition of $\rho$ and of $j,$ 
  $$\begin{aligned}
  \|j(t)-(\rho u)(t)\|_{L^1}&\leq \int\!\!\!\int f(t)|v-u(t)|\,dv\,dx\\
  &\leq \sqrt{M_0\: E_1(t)},\end{aligned}$$
which completes the proof of \eqref{eq:conv1}. 
\smallbreak
 Finally,  integrating with respect to the kinetic and time variables gives for all $t\geq0,$
    \begin{equation}\label{eq:eq}
  \rho(t)=\rho_0-\div\biggl(\int_0^t j\,d\tau\biggr)\cdotp\end{equation}
  Now, we have $j=(j-\rho u) +\rho u.$ On the one hand, since $j-\rho u$ satisfies \eqref{eq:conv1}, it is in $L^1(\R_+;L^1).$
  On the other hand, from H\"older and Sobolev inequality,  and \eqref{eq:M0}, we have
  $$  \|(\rho u)(t)\|_{L^1}\leq \|\rho(t)\|_{L^{6/5}}\|u(t)\|_{L^6}\leq C M_0^{5/6} \|\rho(t)\|_{L^\infty}^{1/6}\|\nabla u(t)\|_{L^2}.$$
From Theorem \ref{thm:main1}, the density $\rho$ is uniformly bounded on $\R_+,$
  and $\|\nabla u(t)\|_{L^2}$ decays according to \eqref{eq:E1t}.  Hence $\rho u$ is also  in $L^1(\R_+;L^1).$
  One can thus set  
  $$j_\infty:= \int_0^\infty j\,d\tau\andf \rho_\infty:=\rho_0-\div j_\infty$$
 so that \eqref{eq:eq} implies that
  $$\rho_\infty-\rho(t)=-\div\biggl(\int_t^\infty(j-\rho u)\,d\tau +\int_t^\infty \rho u\,d\tau\biggr)\cdotp$$
  Taking advantage of \eqref{eq:E1t} and of the above bounds for $j-\rho u$ and $\rho u$ yields    \eqref{eq:conv2}.
\medbreak
{Observe moreover that $\rho_\infty$ is in $L^\infty,$ due to the fact that $\|\rho(t)\|_{L^\infty}$ is bounded on $\R_+.$}


\section{The Fujita-Kato theorem for the Vlasov-Navier-Stokes system}\label{s:FK}

This section is devoted to the case where the initial velocity  $u_0$ has only  critical Sobolev regularity $H^{1/2}.$ 
The key step is to establish  that if $\|u_0\|_{\dot H^{1/2}}$ and $E_0$ are small enough, then (VNS) admits a weak solution such that
$u(t)$ remains in $\dot H^{1/2}$ (and controlled by the data)  until time $t=1,$ and  possesses the   log Lipschitz regularity
that is needed for getting  uniqueness  \emph{within the class of Fujita-Kato solutions.}  
All these properties  will be achieved by a bootstrap argument, assuming beforehand that $\|u\|_{L^1(0,1;L^\infty)}$ is small enough. 
As a result, we shall obtain some $t_0\in(0,1)$ for which the smallness condition \eqref{eq:smalldata}
is satisfied by $(f(t_0),u(t_0)).$ Since it is also possible to propagate the regularity $\dot B^{-3/2}_{2,\infty}$ in this setting, 
all the hypotheses of Theorem \ref{thm:main1} are satisfied at time $t_0,$ which allows us to construct
a solution on the interval $[t_0,\infty)$ with the properties stated therein (in particular the optimal decay rates for $E_0$ and~$E_1$). The uniqueness result that will be proved below ensures that
the two solutions coincide on $[t_0,1],$ which will yield the full statement of Theorem \ref{thm:main2}.  


\subsection{The incompressible Navier-Stokes equations with source term}

The key to obtaining \eqref{eq:loglip} is to consider the following Navier-Stokes equations with source term: 
\begin{equation}\label{eq:NS}
\left\{\begin{array}{ll} &\d_t u+u\cdot\nabla u-\Delta u+\nabla P= S\in L^{4/3}(0,T;L^2(\R^3)),\\[1ex]
&\div u=0,\\[1ex]
&u|_{t=0}=u_0\in \dot H^{1/2}.\end{array}\right.\end{equation}
We denote by $u_L$ the solution to the linearized equation, namely $u_L:=u_L^1+u_L^2$ with
\begin{equation}\label{def:uL} u_L^1(t):=e^{t\Delta}u_0\andf u_L^2(t):=\int_0^te^{(t-\tau)\Delta} \cP S(\tau)\,d\tau\end{equation}
where $(e^{t\Delta})_{t>0}$ stands for the heat semi-group and $\cP,$ for the  orthogonal projector on divergence free vector-fields. 
\begin{proposition}\label{p:NS}  
The solution $u_L^1$  to the free heat equation with initial data $u_0$ 
belongs to\footnote{{Here $\wt L^1(\R_+;\dot H^{5/2}):=\wt L^1(\R_+;\dot B^{5/2}_{2,2}) $ is the superspace of $L^1(\R_+;\dot H^{5/2})$ that is defined in Proposition \ref{p:maxreg}.}}
\begin{equation}\label{eq:spaceuL1}\cC(\R_+;\dot H^{1/2})\cap L^2(0,T;\dot B^{3/2}_{2,1})\cap \wt L^1(\R_+;\dot H^{5/2})\end{equation}
 and we have
\begin{equation}\label{eq:uuL1}\|u_L^1\|_{L^\infty(\R_+;\dot H^{1/2})}+\|u_L^1\|_{L^2(0,T;\dot B^{3/2}_{2,1})}
+\|u_L^1\|_{\wt L^1(\R_+;\dot H^{5/2})} \leq C\|u_0\|_{\dot H^{1/2}}.\end{equation}
The solution $u_L^2$ to the heat equation with source term  $S$ belongs to 
\begin{equation}\label{eq:spaceuL2}\cC([0,T];\dot H^{1/2})\cap L^2(0,T;\dot B^{3/2}_{2,1})\cap  L^{4/3}(0,T;\dot H^{2}),\end{equation}
 and satisfies
\begin{equation}\label{eq:uuL2}\|u_L^2\|_{L^\infty(0,T;\dot H^{1/2})}+\|u_L^2\|_{L^2(0,T;\dot B^{3/2}_{2,1})}
+\|u_L^2\|_{L^{4/3}(\R_+;\dot H^{2})} \leq C\|S\|_{L^{4/3}(0,T;L^2)}.\end{equation}
Finally, there exists a universal constant $c$ such that if
\begin{equation}\label{eq:smalluL}\|u_0\|_{\dot H^{1/2}}+ \|S\|_{L^{4/3}(0,T;L^2)} \leq c,\end{equation}
then \eqref{eq:NS} 
admits a unique solution $u$ on $[0,T]$ with 
$$\wt u:=u-u_L^1-u_L^2  \ \hbox{ in }\  \cC([0,T];\dot B^{1/2}_{2,1})\cap L^2(0,T;\dot B^{3/2}_{2,1})\cap L^1(0,T;\dot B^{5/2}_{2,1}),$$
and we have
\begin{equation}\label{eq:wuuL1}\|\wt u\|_{L^\infty(0,T;\dot B^{1/2}_{2,1})}\!+\!\|\wt u\|_{L^2(0,T;\dot B^{3/2}_{2,1})}
\!+\!\|\wt u\|_{L^{1}(0,T;\dot B^{5/2}_{2,1})}\leq C\bigl(\|u_0\|_{\dot H^{1/2}}^2\!+\! \|S\|_{L^{4/3}(0,T;L^2)}^2\bigr)\cdotp\end{equation}
\end{proposition}
\begin{proof}   {The fact that $u_L^1$ and $u_L^2$ belong to the two extremal spaces in \eqref{eq:spaceuL1} and \eqref{eq:spaceuL2}
stems from the maximal regularity properties that are called in Proposition \ref{p:maxreg}. 
As for the improved regularity  $L^2(0,T;\dot B^{3/2}_{2,1})$ (instead of just $L^2(0,T;H^{3/2})$), 
it has been pointed out recently by    Ars\'enio in \cite{Arsenio}, and Arsenio and Houamed in \cite{AH}.}

To prove the last  part of the statement, we observe that $\wt u$ satisfies
$$\d_t\wt u-\Delta\wt u+\nabla P=-u\cdot\nabla u, \quad \div\wt u=0\andf \wt u|_{t=0}=0.$$
Hence, {using  again   Proposition \ref{p:maxreg}}   yields  for all $t\in[0,T],$
$$\|\wt u\|_{L^\infty(0,t;\dot B^{1/2}_{2,1})}+\|\wt u\|_{L^{2}(0,t;\dot B^{3/2}_{2,1})}
+\|\wt u\|_{L^{1}(0,t;\dot B^{5/2}_{2,1})}  \leq C\|u\cdot\nabla u\|_{L^{1}(0,t;\dot B^{1/2}_{2,1})}.$$
Now, taking advantage of {the product law \eqref{eq:productlaw}, of \eqref{eq:equivk}}
 and of Cauchy-Schwarz inequality with respect to the time
variable,  
we obtain 
\begin{multline*}\|\wt u\|_{L^\infty(0,t;\dot B^{1/2}_{2,1})}+\|\wt u\|_{L^{2}(0,t;\dot B^{3/2}_{2,1})}
+\|\wt u\|_{L^{1}(0,t;\dot B^{5/2}_{2,1})}\\  \leq C
\bigl( \|u_L^1\|_{L^{2}(0,t;\dot B^{3/2}_{2,1})}^2+ \|u_L^2\|_{L^{2}(0,t;\dot B^{3/2}_{2,1})}^2+\|\wt u\|_{L^{2} (0,t;\dot B^{3/2}_{2,1})}^2\bigr)\cdotp\end{multline*}
 Provided the smallness condition \eqref{eq:smalluL} is satisfied, 
 it is easy to work out  from this inequality a fixed point scheme so 
as to construct a solution $u$ of \eqref{eq:NS} with the desired properties.
 \end{proof}


\subsection{Construction of a solution on an interval of size $1$} \label{s:size1}

{Consider initial data satisfying the assumptions of Theorem \ref{thm:main2} and  the sequence $(f^n,u^n)_{n\in\N}$ of
smooth approximate solutions of System $(VNS_n)$ corresponding to initial data $(f_0,J_nu_0)$
constructed in Subsection \ref{s:construction}.  Let us denote by $T^n$ their maximal time of existence.
Our aim here is to prove that  if the constant $c_0$ in \eqref{eq:smallness2} is small enough, then $T^n>1$ and 
$(f^n,u^n)$ satisfies the properties of regularity of Theorem \ref{thm:main2} on the time interval $[0,1],$ uniformly with respect to $n.$ }

{Toward this, the main idea is to  use Proposition \ref{p:NS} with $S=j^n-\rho^n u^n.$
In fact, since  $\dot B^{3/2}_{2,1}\hookrightarrow L^\infty$ (see the Appendix) and, obviously $\|J_nu_0\|_{\dot H^{1/2}}\leq \|u_0\|_{\dot H^{1/2}},$ it implies that if, for some $T>0$ we have
\begin{equation}\label{eq:smallness3}\|u_0\|_{\dot H^{1/2}} + \|j^n-\rho^n u^n\|_{L^{4/3}(0,T;L^2)} \leq c,\end{equation}
then
\begin{equation}\label{eq:smallness4} \|u^n\|_{L^2(0,T;L^\infty)}\leq Cc.\end{equation}
Now,  observe that owing to \eqref{eq:L2},  we have for all $T\in(0,T^n),$
\begin{equation}\label{eq:frhou}\|j^n-\rho^n u^n\|_{L^2(0,T\times\R^3)}^2\leq \|\rho^n\|_{L^\infty(0,T\times\R^3)} \int_0^T\!\!\int\!\!\int \rho^n|v-u^n|^2\,dv\,dx\,dt.\end{equation}
Hence, taking advantage of Identity \eqref{eq:energy0} and of Inequality \eqref{eq:rhoNq}, we discover that we have for some constant $C_q$ depending only on $q>3,$ 
\begin{align}\label{eq:jrhou}\nonumber
\|j^n-\rho^n u^n\|_{L^{4/3}(0,T;L^2)}&\leq T^{1/4} \|j^n-\rho^n u^n\|_{L^2(0,T\times\R^3)}\\&\leq  C_q T^{1/4}e^{3T/2}\bigl(1+\|u^n\|_{L^1(0,T;L^\infty)}^{q/2}\bigr) N_q^{1/2}E_0^{1/2}.\end{align}
Fix $T$ be the largest time such that $T\leq 1$ and $T<T^n,$ and  \eqref{eq:smallness4} is satisfied. Assuming with no loss of generality that $Cc\leq1,$
Inequality \eqref{eq:jrhou} yields
$$\|j^n-\rho^n u^n\|_{L^{4/3}(0,T;L^2)}\leq 2e^{3/2}C_q N_q^{1/2} E_0^{1/2}.$$
Therefore, if 
$$\|u_0\|_{\dot H^{1/2}} + 2e^{3/2}C_q N_q^{1/2} E_0^{1/2}<c,$$
then both \eqref{eq:smallness3} and \eqref{eq:smallness4} are satisfied with a strict inequality whenever $T\leq1$ and $T<T^n.$
We observe moreover that \eqref{eq:smallness4} together with  \eqref{eq:rhoNq}, \eqref{eq:jNq} and \eqref{eq:eNq} implies that
\begin{equation}\label{eq:boundv2}\underset{t\in[0,\min(1,T^n))}\sup\|\langle v\rangle^2f^n(t)\|_{L^1(\R^3_v;L^\infty(\R^3_x))}\leq C_qN_{q+2}.\end{equation}
Combining with a bootstrap argument and \eqref{eq:continuation},  one can conclude that $T^n>1$ and that all the above inequalities hold true 
on $[0,1]$. Of course, we still have \eqref{eq:boundfn} as well as the energy balance \eqref{eq:energyb}. 
All this allows to repeat the compactness argument of \cite{BDGM} and pass to the limit
in System $(VNS_n).$ The obtained solution $(f,u)$ satisfies in addition  \eqref{eq:smallness3} and \eqref{eq:smallness4} with $T=1$
and thus \eqref{eq:uuL1}, \eqref{eq:uuL2} and \eqref{eq:wuuL1} with $T=1.$
\smallbreak
Let us next establish the Log Lipschitz Inequality \eqref{eq:loglip} on $[0,1].$}
{To do so, we split $u$ into $\wt u + u_L^1 + u_L^2$ where $u_L^1$ and $u_L^2$ have been defined in \eqref{def:uL} 
with $S=j-\rho u.$   By the previous construction and \eqref{eq:jrhou} it is known that
  $S$ is small in $L^{4/3}([0,1];L^2).$ 
Hence,  Proposition \ref{p:NS}  gives us that  $u_L^1$ and $\wt u$ are in  $\wt L^1(0,1;\dot H^{5/2}).$  
These two functions are thus  in $L^1(0,T;C_{\wt\omega_\eta})$ for all 
$\eta\in(0,1/2)$ since $\wt L^1(0,T;\dot H^{5/2})$ is embedded in $L^1(0,T;C_{\wt\omega_\eta})$
(see \cite[p. 237]{BCD}).} 

{Using Inequality \eqref{eq:jrhou}  for $(f,u),$ we see that   $S\in L^2([0,1];L^2),$ which ensures that
 $\nabla^2u_L^2\in L^{2}(0,1;L^2)$ since  $u_L^2|_{t=0}=0$ and 
$$\d_tu_L^2-\Delta u_L^2=\cP S\in L^2(0,1;L^2).$$
Moreover, the fact that   $u\in L^2(0,1;L^\infty)$ and  $\rho, j$ are bounded (a consequence of \eqref{eq:rhoNq} and \eqref{eq:jNq})
implies that  we also have  $S\in L^2(0,1;L^\infty).$ 
Hence, by H\"older inequality and the fact
that the Leray projector $\cP$ maps $L^4$ to itself,  we have  $\cP S\in L^2(0,1;L^4).$ 
Using  the standard maximal regularity result for linear parabolic equations (see e.g.  \cite{Amann}) thus
ensures that $\nabla^2 u_L^2 \in L^2(0,1;L^4).$ 
Now, using the following Gagliardo-Nirenberg inequality:
$$\|z\|_{L^\infty}\leq C\sqrt{\|\nabla z\|_{L^2}\|\nabla z\|_{L^4}}\,,$$
we get  $\nabla u_L^2\in L^2(0,1;L^\infty)$ which is better than \eqref{eq:loglip}.
This completes the proof of the existence of a solution satisfying the properties listed in Theorem \ref{thm:main2}, until time $1.$
The reader is invited to read Subsection \ref{s:continuation} 
 to discover how to extend the solution for all time. }


\subsection{Uniqueness} 

Our aim is to establish the uniqueness of solutions satisfying the properties of regularity stated in Theorem \ref{thm:main2}. 
Compared to Theorem \ref{thm:main1}, the difficulty is that  having  $u_0$ in $\dot H^{1/2}$
does not imply that the corresponding solution $(f,u,P)$ of (VNS) satisfies  $\nabla u\in L^1_{loc}(\R_+;L^\infty)$  for the simple reason that 
even $e^{t\Delta} u_0$ does not have this property. 
It turns out however that as in the two-dimensional case treated before in \cite{HMMM}, the 
 Log Lipschitz regularity \eqref{eq:loglip} ensures uniqueness in the framework of critical regularity solutions. 
\medbreak
So, assume that we are given two solutions $(f_1,u_1,P_1)$ and  $(f_2,u_2,P_2)$ 
fulfilling the properties listed in Theorem \ref{thm:main2} and \eqref{eq:loglip}. 
Even if not stated explicitly there, this implies that $\rho_i,$ $j_i$ and $m_{2,i}$ are in $L^\infty(0,T;L^\infty)$ for $i=1,2$
(use  \eqref{eq:rhoNq},  \eqref{eq:jNq} and \eqref{eq:eNq}). 
\medbreak
We start again from \eqref{eq:du0}, bounding 
the first term of the right-hand side according  to \eqref{eq:uniq1} and still splitting 
$\dF$ (defined in \eqref{eq:dF}) into  $A_1+A_2+A_3+A_4.$ Term $A_2$
has to be treated differently as before as we do not have any control on $\|\nabla u_1\|_{L^\infty}.$ 
The second change is that when bounding $\dZ,$ one cannot  use   \eqref{eq:f01}.
Now, by using Cauchy-Schwarz inequality and \eqref{eq:formularho}  with $Z=Z_1,$ we get
$$A_2\leq \|\sqrt{\rho_1}\,\du\|_{L^2}\biggl(\int\!\!\!\int f_0(x,v)|u_1(t,X_1(t,x,v))-u_1(t,X_2(t,x,v))|^2\,dv\,dx\biggr)^{1/2}.$$
Let $\omega_\eta:r\mapsto r(2-\log r)^{2-2\eta}.$ Remembering \eqref{eq:loglip}, one can assert that there exists an integrable function $\alpha$ on $[0,T]$ such that
$$\int\!\!\!\int |u_1(t,X_1(t))-u_1(t,X_2(t))|^2\,f_0\,dv\,dx\leq \alpha^2(t) 
\int\!\!\!\int  \omega_\eta(|\dX(t)|^2) \, f_0\,dv\,dx.$$
Observe that the function $\omega_\eta$ is concave and increasing on some nontrivial interval $[0,r_0]$ and 
that we have $|\dX(t,x,v)|^2\leq r_0$ pointwise for all $t\in[0,T]$ and $(x,v)\in\R^3\times\R^3$ provided $T>0$ is small enough.
Indeed, relation \eqref{eq:X} implies that
$$\dX(t,x,v)=\int_0^t \bigl(e^{\tau-t}-1\bigr) \bigl(u_2(\tau,X_2(\tau,x,v))- u_1(\tau,X_1(\tau,x,v))\bigr)\,d\tau,$$
and thus 
$$|\dX(t,x,v)|\leq \sqrt t\,\bigl(\|u_1\|_{L^2(0,t;L^\infty)} + \|u_2\|_{L^2(0,t;L^\infty)}\bigr)\cdotp $$
Therefore, assuming with  no loss of generality  that  the total mass $M_0$ defined in \eqref{eq:M0} is equal to $1,$
Jensen inequality allows us to get
$$A_2(t)\leq \alpha(t) \|\sqrt{\rho_1}\,\du\|_{L^2} \bigl(\omega_\eta(\|\sqrt{f_0}\,\dX\|_{L^2}^2)\bigr)^{1/2}.$$ 
Similarly,  the term $B_1$ that has been defined in  \eqref{eq:f01} satisfies
$$\begin{aligned}
B_1(t)&\leq
\|\sqrt{f_0}\,\dV(t)\|_{L^2} \biggl(\int\!\!\!\int |u_1(t,X_1(t))-u_1(t,X_2(t))|^2\,f_0\,dv\,dx\biggr)^{1/2}\\
&\leq \alpha(t)\|\sqrt{f_0}\,\dV(t)\|_{L^2}. 
\end{aligned}$$
 In the end we get the following inequality for $Y:=\|\du\|_{L^2}^2+\|\sqrt{f_0}\,\dZ\|_{L^2}^2$:
\begin{multline*} \frac d{dt}Y+
\|\nabla\du\|_{L^2}^2\lesssim (1+\|\rho_1\|_{L^\infty})^{1/2} \alpha \,\omega_\eta(Y)
+\bigl(1+\|\nabla u_2\|_{L^3}^2+\|(\rho_1,\rho_2)\|_{L^\infty}\|u_1\|_{L^\infty}^2 \\+\|(m_{2,1},m_{2,2})\|_{L^\infty}
+\|\rho_2\|_{L^\infty}\|(u_1,u_2)\|_{L^2(0,t;L^\infty)}\bigr)Y. 
\end{multline*}
Note that  $\rho_1, \rho_2, m_{2,1}, m_{2,2}$  are bounded, $u_1, u_2$ are in $L^2(0,T;L^\infty)$ 
and $\alpha$ is integrable. Since the modulus of continuity $\omega_\eta$ satisfies Osgood's condition, that is 
$$\int_0^{r_0} \omega_\eta^{-1}(r)\,dr=\infty \ \hbox{ for some }\  r_0>0,$$
 one can conclude that $Y\equiv0$ on $[0,T]$ (apply e.g. \cite[Lem. 3.4]{BCD}). 
 This completes the proof of the uniqueness part of Theorem \ref{thm:main2}.
  

{\subsection{Continuation of the solution} \label{s:continuation}
So far, we constructed a critical regularity solution $(f,u,P)$ on the time interval $[0,1].$ To continue this solution into a global one, 
the key is to prove that there exists some $t_0\in(0,1)$ so that the smallness property \eqref{eq:smalldata} is satisfied 
by  $f(t_0)$ and $u(t_0).$  Denoting by $(\wt f,\wt u,\wt P)$ the corresponding smoother solution produced by Theorem \ref{thm:main1}, one can 
then leverage the uniqueness result we have just proved so as to justify that $(f,u,P)=(\wt f,\wt u,\wt P)$ on $[t_0,1].$
Hence,  $(\wt f,\wt u,\wt P)$ is a global extension of $(f,u,P)$ that satisfies all the asymptotic properties
stated in Theorem \ref{thm:main1}. 
\medbreak
To check  \eqref{eq:smalldata}, we need to find some $t_0\in(0,1)$ so  that $u(t_0)$ is in $L^1$ (with some control on the norm) 
and that $E_{1}(t_0)$ is small (so far, we do not know whether it is finite). 
As seen before, we do not really need to control the $L^1$ norm of $u,$ but rather 
 regularity $\dot B^{-3/2}_{2,\infty}.$ 
 Now,  Inequality \eqref{eq:3/2}, the embeddings that follow and Cauchy-Schwarz inequality ensure that
\begin{align}\label{eq:bound2}\nonumber\underset{t\in[0,1]}\sup \|u(t)\|_{\dot B^{-\frac32}_{2,\infty}} &\leq  \|u_0\|_{\dot B^{-\frac32}_{2,\infty}}
+\|u\cdot\nabla u\|_{L^2(0,1;L^1(\R^3))} 
+ \biggl\|\Int  {(u-v)f}\,dv\biggr\|_{L^2(0,1;L^1(\R^3))}\\
 &\leq  \|u_0\|_{\dot B^{-\frac32}_{2,\infty}}+ E_{0,0} + \sqrt{M_0E_{0,0}}\:.
\end{align}
Remember in addition that $f$ satisfies  \eqref{eq:boundv2}.
Let us define the smallness constant $c_0$ in \eqref{eq:smalldata} according to \eqref{eq:boundv2} and
\eqref{eq:bound2}, and assume that
 $$E_{0,0}+\||v|^2f_0\|_{L^1_{x,v}}\leq \frac 12c_0.$$ The energy balance \eqref{eq:energyb} and the fact that $E_1=D_0$ then ensure that
 $$
 \int_0^1 E_{1}\,d\tau \leq \frac12 c_0.$$ Hence, 
 there exists some $t_0\in(0,1)$ so that $E_1(t_0)\leq c_0/2,$ and
 Condition \eqref{eq:smalldata} is thus satisfied at time $t=t_0.$ 
 This completes the global existence part of Theorem \ref{thm:main2}, with optimal time decay. }


\subsection{The case of an initial velocity in $\dot B^{1/2}_{2,1}$} \label{ss:critic}

Here we justify Remark \ref{r:main2} pertaining to  the case where  the initial data $(f_0,u_0)$  satisfy \eqref{eq:f0} and 
$$\|u_0\|_{B^{1/2}_{2,1}(\R^3)} + \int\!\!\!\int |v|^2 f_0\,dv\,dx\leq c\ll1.$$
Compared to the case $u_0\in\dot H^{1/2},$ 
the key  difference lies in the fact that  $u_L^1=e^{t\Delta}u_0$   satisfies 
$$u_L^1\in \cC_b(\R_+;\dot B^{1/2}_{2,1})\cap L^1(\R_+;\dot B^{5/2}_{2,1}),$$
which implies that $\nabla u_L^1\in  L^1(\R_+;L^\infty)$ with 
$$
\int_0^\infty \|\nabla u_L^1\|_{L^\infty} \leq Cc.$$
Consequently, one can bound $\rho$ and $m_2$ in $L^\infty(0,T\times\R^3)$ (and thus $j$) according to \eqref{eq:boundrho} and \eqref{eq:bounde}
whenever \eqref{eq:Lip} holds true. 
Now, these latter properties enable us to bound 
$S:=j-\rho u$ in $L^2_{loc}(\R_+;L^\infty)$ and one can  argue like in Subsection \ref{s:size1}
to prove \eqref{eq:Lip} on, say, the time interval $[0,1].$ 
The rest of the proof  goes as before. Compared to Theorem \ref{thm:main2}, 
the gain is that we do not need to assume any longer that $N_q(f_0)<\infty$ for some $q>5.$

\medbreak\noindent{\bf Conflict of interest.}
The author has no financial or non-financial interest to disclose. The author has no conflict of interest to declare that are relevant to the content of this article.


{\medbreak\noindent{\bf Acknowledgments.}  
The  author is partially  supported by \emph{Institut Universitaire de France}. He
 is grateful to the anonymous referees, whose many pertinent comments have greatly improved the presentation of the results in this article.}


\appendix\section{The control of the flow, and applications} \label{s:AppendixA}

Most of the regularity estimates that were used in this paper require boundedness of the density $\rho$ and, to a lesser extent, of the momentum $j$ and of the  energy distribution $m_2$.
In order to derive these bounds, we argue as in \cite{HMM}, computing
$\rho,$ $j$  and $m_2$ in terms of  the characteristics $Z:=(X,V)$ associated
 to the $f$ equation. To do so, looking at  $(t,x,v)\in \R_+\times \R^3\times\R^3$ as parameters, 
 we consider the following  system of ODEs:\begin{equation}\label{eq:car} \left\{\begin{aligned}
	\d_sX(s;t,x,v)&= V(s;t,x,v)\\
	\d_sV(s;t,x,v)&=u(s; X(s;t,x,v))- V(s;t,x,v)\\
	X(t;t,x,v)=x&\andf V(t;t,x,v)=v.	
	\end{aligned}\right.	\end{equation}
	In other words,  $Z=(X,V)$ is the flow of the time-dependent vector-field 
	$$F(t,x,v)=(v,u(t,x)-v),\qquad (t,x,v)\in\R_+\times\R^3\times\R^3.$$
The  Cauchy-Lipschitz theorem ensures that 
	for \eqref{eq:car} to be solvable, it suffices that $u\in L^1_{loc}(\R_+;W^{1,\infty}).$ 
		Furthermore, since ${\rm Div}_{x,v} F=-3,$ we have 
$$\det D_{x,v} Z(s;t,x,v)= e^{3(t-s)}.$$
The solution  of the first equation  of (VNS)  then reads:	
	\begin{equation}\label{eq:formularho}
	f(t,x,v)=e^{3t} f_0\bigl(X(0;t,x,v),V(0;t,x,v)\bigr),\end{equation}
	and thus 
	$$	\rho(t,x)=\int f(t,x,v)\,dv= e^{3t}\int  f_0\bigl(X(0;t,x,v),V(0;t,x,v)\bigr)\,dv.	$$
	According to \cite[Lem. 4.4]{HMM},  if the Lipschitz condition \eqref{eq:Lip} is satisfied
		then, for all $t\in[0,T]$ and $x\in\R^3,$  the map $\Gamma_{t,x}: v\mapsto V(0;t,x,v)$ is a bilipschitz homeomorphism  on $\R^3$
	satisfying
	\begin{equation}\label{eq:diff}
	\det D_v \Gamma_{t,x}(v)\geq e^{3t}/2.
	\end{equation}
	 	Now, performing the change of variable $w=\Gamma_{t,x}(v)$ in \eqref{eq:formularho}, we get
		$$
		\rho(t,x)= e^{3t}\int f_0(X(0;t,x,\Gamma_{t,x}^{-1}(w)),w)\det D_w\Gamma^{-1}_{t,x}(w)\,dw,
		$$
		and thus, owing to \eqref{eq:diff},
		\begin{equation}\label{eq:boundrho} |\rho(t,x)| 
\leq	2\|f_0\|_{L^1(\R^3_v;L^\infty(\R^3_x))}.		\end{equation}
		Similarly, we have
		$$\begin{aligned}
		m_2(t,x)&=e^{3t}\int|v|^2  f_0\bigl(X(0;t,x,v),V(0;t,x,v)\bigr)\,dv\\
		&=e^{3t}\int  |\Gamma_{t,x}^{-1}(w)|^2 f_0(X(0;t,x,\Gamma_{t,x}^{-1}(w)),w)\det D_w\Gamma^{-1}_{t,x}(w)\,dw.
		\end{aligned}$$
		In order to bound  $|\Gamma_{t,x}^{-1}(w)|^2,$ we integrate the second line of \eqref{eq:car} on $[0,t],$ getting
		$$	\Gamma_{t,x}^{-1}(w)=e^{-t}\biggl(w+\int_0^t e^s u\bigl(s,X(s;t,x,\Gamma_{t,x}^{-1}(w))\bigr)ds\biggr),	$$
whence 	 
$$\bigl|\Gamma^{-1}_{t,x}(w)\bigr|\leq e^{-t}|w| +\int_0^t e^{s-t} \|u(s)\|_{L^\infty}\,ds.$$
	Therefore, using again \eqref{eq:diff}, it is easy to get
	$$
	m_2(t,x)\leq 4e^{-2t}\||v|^2 f_0\|_{L^1(\R^3_v;L^\infty(\R^3_x))} +4\biggl(\int_0^t e^{s-t}\|u(s)\|_{L^\infty}\,ds\biggr)^2
	\| f_0\|_{L^1(\R^3_v;L^\infty(\R^3_x))}.	$$
	{Now, we have, according to Gagliardo-Nirenberg inequality
	$$	\|u\|_{L^\infty}\leq C\|u\|_{L^2}^{2/5}\|\nabla u\|_{L^\infty}^{3/5}.$$
	Hence, using H\"older inequality, we deduce that
	 $$\int_0^t e^{s-t}\|u(s)\|_{L^\infty}\,ds\leq C\biggl(\underset{s\in[0,t]}\sup \|u(s)\|_{L^2}\biggr)^{2/5}
	 \biggl(\int_0^t\|\nabla u\|_{L^\infty}\,ds\biggr)^{3/5}.$$
	 Since \eqref{eq:Lip} is satisfied, remembering  \eqref{eq:energyb} 
	 allows to conclude that
	\begin{equation}\label{eq:bounde}
	m_2(t,x)\leq 4e^{-2t}\||v|^2 f_0\|_{L^1(\R^3_v;L^\infty(\R^3_x))} + C\delta^{6/5} E_{0,0}^{2/5} \| f_0\|_{L^1(\R^3_v;L^\infty(\R^3_x))}.\end{equation}}
	Note that the above calculations actually give a bound of $\langle v\rangle ^2f$ in $L^\infty(\R_+;L^1(\R^3_v;L^\infty(\R^3_x))).$
	\medbreak
		We now look at the situation  where $u$ only satisfies the  log Lipschitz property \eqref{eq:loglip}
	and is in $L^1_{loc}(\R_+;L^\infty).$ 
	Then, Osgood lemma (see e.g.  \cite[Chap. 3]{BCD}) still guarantees 
	that  System \eqref{eq:car} has a unique solution. Furthermore,  direct manipulations reveal that 	\begin{align}\label{eq:X}
	X(s;t,x,v)&=x+(1-e^{t-s})v +\int_s^t \bigl(e^{\tau-s}-1\bigr) u(\tau, X(\tau;t,x,v))\,d\tau,\\
\label{eq:V}	V(s;t,x,v)&=e^{t-s} v - \int_s^t e^{\tau-s}  u(\tau, X(\tau;t,x,v))\,d\tau.\end{align}	
	This implies that for all $t\in[0,T],$ we have
	$$|v|\leq  e^{-t}|V(0;t,x,v)|+\int_0^t e^{\tau-t}\|u(\tau)\|_{L^\infty}\,d\tau.$$
	Hence, since $f(t,x,v)=e^{3t} f_0(Z(0;t,x,v)),$ we have  for all $q>3,$
$$\begin{aligned}
\rho(t,x)=\int f(t,x,v)\,dv &=  e^{3t}\int \langle v\rangle^q  f_0(Z(0;t,x,v)) \,  \langle v\rangle^{-q} \,dv\\
&\leq C_q e^{3t}\biggl(\underset{v\in\R^3}\sup  \langle v\rangle^q  f_0(Z(0;t,x,v)) \biggr)\\
&\leq C_q e^{3t} \biggl(e^{-tq}\underset{v\in\R^3}\sup   \langle V(0;t,x,v)\rangle^q  f_0(Z(0;t,x,v)) \\&\hspace{1cm}
+ \biggl(\int_0^t e^{\tau-t} \|u(\tau)\|_{L^\infty}\,d\tau\biggr)^q \underset{v\in\R^3}\sup{f_0(Z(0;t,x,v))}\biggr)\cdotp
\end{aligned}
$$
Denoting $N_q=\sup_{(x,v)\in\R^3\times\R^3} \langle v\rangle^q f_0(x,v),$ one can conclude that 
\begin{equation} 
\label{eq:rhoNq} \|\rho(t)\|_{L^\infty} \leq C_q e^{3t} \bigl(1+ \|u\|_{L^1(0,t;L^\infty)}\bigr)^q N_q.\end{equation}
Similarly, we have, if $q>3,$
$$\begin{aligned}
j(t,x)=\int v f(t,x,v)\,dv &=  e^{3t}\int  v\langle v\rangle^q  f_0(Z(0;t,x,v)) \,  \langle v\rangle^{-q} \,dv\\
&\leq C_{q+1} e^{3t}\biggl(\underset{v\in\R^3}\sup  \langle v\rangle^{q+1}  f_0(Z(0;t,x,v)) \biggr)\cdotp
\end{aligned}
$$
Hence, 
\begin{equation} 
\label{eq:jNq} \|j(t)\|_{L^\infty} \leq C_{q+1} e^{3t} \bigl(1+ \|u\|_{L^1(0,t;L^\infty)}\bigr)^{q+1} N_{q+1}.\end{equation}
The same argument leads to 
\begin{equation} 
\label{eq:eNq} \|m_2(t)\|_{L^\infty} \leq C_{q+2} e^{3t} \bigl(1+ \|u\|_{L^1(0,t;L^\infty)}\bigr)^{q+2} N_{q+2}.\end{equation}
{Again, note that  the same  bounds hold true for the norms of  $f(t),$ $vf(t)$ and $|v|^2f(t)$ in $L^1(\R^3_v;L^\infty(\R^3_x)).$}


\section{{Besov spaces, interpolation and parabolic maximal regularity}}	\label{s:AppendixB}

{For reader's convenience, we present some properties of the homogeneous Besov spaces, 
as well as some interpolation results and maximal regularity estimates.
\subsection{Besov spaces}
We here give a short presentation of Besov spaces.  More details may be found in e.g.  \cite[Chap. 2]{BCD}.
\smallbreak
The  most common definition is based on  a \emph{homogeneous Littlewood-Paley decomposition}: 
we fix some smooth nonincreasing function $\chi:\R_+\to [0,1]$ supported in $[0,2]$ with $\chi\equiv1$ on $[0,1],$
then set $\varphi(\xi):= \chi(|\xi|/2)-\chi(|\xi|)$ for all $\xi\in\R^d.$ By construction, we have
\begin{equation}\label{eq:LP}\sum_{j\in\Z} \varphi(2^{-j}\xi)= 1\quad\hbox{for all }\ \xi\in\R^d\setminus\{0\}.\end{equation}
For all $j\in\Z$ we define the spectral cut-off operator $\ddj$ acting on tempered distributions by 
$$\ddj z:= \varphi(2^{-j}D) z:= \mathcal F^{-1}(\varphi(2^{-j}\cdot) \mathcal F z),$$
where $\mathcal F$ denotes the Fourier transform on $\R^d.$
\medbreak
It is not difficult to check that we have, in the tempered distributions sense,
\begin{equation}\label{eq:Lpdecompo} z=\sum_{j\in\Z} \ddj z\quad\hbox{modulo polynomials.}\end{equation}
Granted with a homogeneous Littlewood-Paley decomposition, one can 
define homogeneous Besov (semi)-norms as follows for all $1\leq p,r\leq\infty$ and $s\in\R$:
\begin{equation}\label{def:besov}\|z\|_{\dot B^s_{p,r}}:=\Bigl\| 2^{js}\|\ddj z\|_{L^p(\R^d)}\Bigr\|_{\ell^r(\Z)}.\end{equation}	
Note that  $\|P\|_{\dot B^s_{p,r}}=0$ for any polynomial $P.$ In order to have a  normed space endowed with $\|\cdot\|_{\dot B^s_{p,r}},$
we here adopt the viewpoint of \cite{BCD}: we define the \emph{homogeneous Besov space $\dot B^s_{p,r}$} to be the set of tempered distributions $z$
such that 
$$\|z\|_{\dot B^s_{p,r}}<\infty\andf \underset{j\to-\infty}\lim \|\chi(2^{-j}D)z\|_{L^\infty}=0.$$
Let us underline that all the above definitions are independent of the choice of
the cut-off function $\chi,$ up to change of norms into equivalent norms. 
\medbreak
In the present paper, we always take $p=2$ and $r\in\{1,2,\infty\}.$
Then, it can be easily seen from \eqref{eq:LP} that for all $s\in\R,$ we have
\begin{equation}\label{eq:equivHs}
 \|u\|_{\dot H^s}\simeq \|u\|_{\dot B^s_{2,2}}\andf \|u\|_{\dot B^s_{2,\infty}}\leq  \|u\|_{\dot B^s_{2,2}}\leq  \|u\|_{\dot B^s_{2,1}}. \end{equation}}
{From Fourier-Plancherel theorem, we can see  that for all  $s\in\R$ and $r\in[1,\infty],$ we have
\begin{equation}\label{eq:equivk} \|\nabla z\|_{\dot B^{s-1}_{2,r}}\simeq \|z\|_{\dot B^s_{2,r}},\end{equation}
and that the Leray projector $\cP$ satisfies
\begin{equation}\label{eq:cP}
\|\cP z\|_{\dot B^s_{2,r}}\leq \|z\|_{\dot B^s_{2,r}}.\end{equation}}
{One of the advantages of  the Besov space
$\dot B^{d/2}_{2,1}(\R^d)$ compared with the Sobolev space  $\dot H^{d/2}(\R^d)$ is its good behavior with respect to embedding and product laws:
it is embedded in the set of bounded continuous functions and satisfies 
\begin{equation}\label{eq:productlaw} 
\| z_1\,z_2\|_{\dot B^{s}_{2,1}}\leq C_{s,d} \|z_1\|_{\dot B^{d/2}_{2,1}}\|z_2\|_{\dot B^{s}_{2,1}},\qquad -d/2<s\leq d/2.\end{equation}}


{\subsection{Parabolic maximal regularity}
We here give a short presentation of maximal regularity estimates  for the heat equation
\begin{equation}\label{eq:heat}
\left\{\begin{aligned} &z_t -\Delta z=g\quad&\hbox{in }\  \R_+\times\R^d,\\
 &z|_{t=0}=z_0\quad&\hbox{in }\ \R^d.\end{aligned}\right.\end{equation}
Let us first recall the classical maximal regularity result (see e.g. \cite{Amann}).
\begin{proposition} \label{p:standardmaxreg}  Let  $1<p,r<\infty.$ For any $z_0\in \dot B^{2-2/r}_{p,r}(\R^d)$  and 
$g\in L^r(0,T;L^r(\R^d)),$ Equation \eqref{eq:heat} has a unique solution $z\in\cC([0,T];\dot B^{2-2/r}_{p,r}(\R^d))$ such that
$z_t$ and $\nabla^2_xz$ are in $L^r(0,T;L^p(\R^d)),$ and we have 
$$\|z\|_{L^\infty(0,T;\dot B^{2-2/r}_{p,r})}+\|z_t\|_{L^r(0,T;L^p)} +\|\nabla^2_xz\|_{L^r(0,T;L^p)} \leq C_{p,r,d}\Bigl(\|z_0\|_{\dot B^{2-2/r}_{p,r}}+ \|g\|_{L^r(0,T;L^p)}\Bigr)\cdotp$$
\end{proposition}}

{In the case of initial data and source terms in Besov spaces,  we have the following a priori estimates for 
\eqref{eq:heat},  first observed by Chemin and Lerner \cite{ChL} in the particular case of Sobolev spaces then 
extended by Chemin in \cite{Che}:
 \begin{proposition}\label{p:maxreg} There exists a constant $C$ depending only on the dimension $d$ such that
 for any $1\leq p, r \leq\infty,$ $s\in\R,$  $1\leq \rho_1\leq\rho_2\leq\infty$ and $T\geq0,$  we have
 \begin{equation}\label{eq:maxreg}\|z\|_{\wt L^{\rho_1}_T(\dot B^{s+2/\rho_1}_{p,r})} \leq C\Bigl(\|z_0\|_{\dot B^s_{p,r}}
 +\|g\|_{\wt L^{\rho_2}_T(\dot B^{s-2+2/\rho_2}_{p,r})}\Bigr),\end{equation}
 where we used the notation $\|a\|_{\wt L^\rho_T(\dot B^\sigma_{p,r})}:=
 \Bigl\|2^{j\sigma}\bigl\|\ddj a\bigr\|_{L^\rho(0,T;L^p(\R^d))}\Bigr\|_{\ell^r(\Z)}\cdotp$
  \end{proposition}
\begin{proof} For reader's convenience, let us give the proof in the case $p=2$ (the only one that we used here). The starting point is
that  for some absolute constant $c,$ we have 
\begin{equation}\label{eq:tdj}
\bigl\|e^{t\Delta} \ddj z\|_{L^2}\leq e^{-ct2^{2j}}\|\ddj z\|_{L^2}.\end{equation}
This stems from computing  $e^{t\Delta} \ddj z$ in the Fourier space, and  Fourier-Plancherel theorem. 
\smallbreak
Now, localizing the heat equation by means of $\ddj$ gives for all $j\in\Z,$
$$\left\{\begin{aligned} &\d_t\ddj z -\Delta \ddj z=\ddj g,\\
 &\ddj z|_{t=0}=\ddj z_0.\end{aligned}\right.$$
Hence, the Duhamel formula gives
$$\ddj z(t)= e^{t\Delta}\ddj z_0+\int_0^t e^{(t-\tau)\Delta}\ddj g(\tau)\,d\tau.$$
Taking the $L^2$ norm and using Minkowski inequality, \eqref{eq:tdj} and a convolution inequality 
with respect to the time variable and taking in the end the $\ell^r(\Z)$ norm
 yields Inequality \eqref{eq:maxreg}  in the particular case where $p=2.$\end{proof} }

{It is easy to compare the norms defined in the above proposition with more classical ones: 
from  Minkowski inequality, we have:
\begin{equation} \label{eq:Minko}\begin{aligned}
 \|a\|_{\wt L^\rho_T(\dot B^\sigma_{p,r})}&\leq \|a\|_{L^\rho_T(\dot B^\sigma_{p,r})}
 &\hbox{if }\  r\geq\rho,\\
 \|a\|_{L^\rho_T(\dot B^\sigma_{p,r})}&\leq  \|a\|_{\wt L^\rho_T(\dot B^\sigma_{p,r})}
 &\hbox{if }\  r\leq\rho.
\end{aligned}\end{equation}}


{\subsection{Interpolation and embedding}
Homogeneous Besov norms satisfy a number of interpolation inequalities.
 Let us just mention the following ones which  can be easily proved by splitting 
the sum in \eqref{eq:Lpdecompo} into $j\leq N$ and $j>N,$  and then optimizing $N$:
\begin{equation}\label{eq:interpo3} 
\|u\|_{\dot B^1_{2,2}}\lesssim \|u\|_{\dot B^2_{2,2}}^{\frac{\sigma +1}{\sigma+2}} \|u\|_{\dot B^{-\sigma}_{2,\infty}}^{\frac{1}{\sigma+2}}\andf  
\|u\|_{\dot B^0_{2,2}}\lesssim \|u\|_{\dot B^1_{2,2}}^{\frac{\sigma +1}{\sigma+2}} \|u\|_{\dot B^{-\sigma-1}_{2,\infty}}^{\frac{1}{\sigma+2}},\qquad \sigma>-1.\end{equation}
By the same method, one can obtain similar estimates for spaces $\wt L_T^\rho(\dot B^s_{2,r})$: 
the time exponent behaves according to H\"older inequality. In particular, this gives \eqref{eq:interpo4}. 
\medbreak
Homogeneous Besov norms with negative regularity indices may be equivalently defined in terms of the heat
flow: we have
\begin{equation}\label{eq:besovheat}
\|z\|_{\dot B^{-\sigma}_{p,r}(\R^d)}\simeq \bigl\|t^{\sigma/2}\|e^{t\Delta}z\|_{L^p(\R^d)}\bigr\|_{L^r(\R_+;\frac{dt}t)}.
\end{equation}
The following example  played an important role in the paper:
$$\|z\|_{\dot B^{-3/2}_{2,\infty}(\R^3)}\simeq  \underset{t>0}\sup \bigl\|t^{3/4}e^{t\Delta} z\bigr\|_{L^2(\R^3)}.$$
From the  Gaussian estimates of the heat kernel, one can easily show that the right-hand 
side of the above inequality is bounded by $\|z\|_{L^1(\R^3)},$ whence  the continuous  
embedding: \begin{equation}\label{eq:embed}
L^1(\R^3)\hookrightarrow \dot B^{-3/2}_{2,\infty}(\R^3).\end{equation}
Similarly, we have
$$\|z\|_{\dot B^{-1/2}_{2,\infty}(\R^3)}\simeq \underset{t>0}\sup \bigl\|t^{1/4}e^{t\Delta z}\bigr\|_{L^2(\R^3)}\lesssim
\|z\|_{L^{6/5}(\R^3)}.$$}
Lorentz spaces can be defined on any measure space $(X,\mu)$
 via  real interpolation between the classical Lebesgue spaces, as follows:
\begin{equation}\label{eq:lorentz}  L_{p,r}(X,\mu) := (L_\infty(X,\mu),L_1(X,\mu))_{1/p,r}
\ \hbox{ for }\ p\in (1,\infty)\andf r \in [1,\infty].
\end{equation}
According to the reiteration theorem (see \cite{Be}), we have as well
\begin{equation}\label{eq:reiteration} L_{p,r}(X,\mu) := (L_{p_0}(X,\mu),L_{p_1}(X,\mu))_{\theta,r}
\ \hbox{ for all } r\in[1,\infty],
\end{equation}
whenever  $p_0\not= p_1$ and $\frac1p=\frac{1-\theta}{p_0}+\frac\theta{p_1}$ with $\theta\in(0,1).$
\smallbreak 
The following classical properties may be found in, e.g.,~\cite{grafakos}:
\begin{itemize}
\item[--] Embedding : $L_{p,r_1}\hookrightarrow L_{p,r_2}$ if $r_1\leq r_2,$ 
and $L_{p,p}=L_p.$ 
\item[--] H\"older inequality :  for $1<p,p_1,p_2<\infty$ and
$1\leq r,r_1,r_2\leq\infty$, we have
\begin{equation}\label{eq:holder}
\|fg\|_{L_{p,r}}\lesssim \|f\|_{L_{p_1,r_1}} \|g\|_{L_{p_2,r_2}}\with
\frac 1p=\frac1{p_1}+\frac1{p_2}\andf 
\frac 1r=\frac1{r_1}+\frac1{r_2}\cdotp
\end{equation}
\end{itemize}
The Lorentz space $L^{d,1}(\R^d)$ has also the following interesting embedding property:
\begin{equation}\label{eq:emdedLorentz}\|z\|_{L^\infty(\R^d)}\leq C\|\nabla z\|_{L^{d,1}(\R^d)}.\end{equation}
One of the reasons for the success of the theory of interpolation is its application to the continuity of linear applications. We have used the following version (see \cite{Be}):
\begin{proposition}\label{p:interpo} Assume that $1\leq p_0<p_1\leq\infty.$
Let  $\Phi:L^{p_0}+L^{p_1}\to L^{p_0}+L^{p_1}$ be a linear mapping such that 
$\Phi|_{L^{p_i}}$ is continuous  from $L^{p_i}$ to $L^{p_i}$ for $i=0,1.$ 
Then, $\Phi:L^{p,r}\to L^{p,r}$ is continuous for  any $p\in(p_0,p_1)$  and $r\in[1,\infty],$ 
 \end{proposition}
As example, let us consider the Leray projector $\cP.$ As it is continuous on $L^p$ for any $p\in(1,\infty),$ it is also  continuous 
on $L^{3,1},$ a property that we used repeatedly in this paper.

   \begin{small}	 

\end{small}

\end{document}